\documentclass{siamart190516}
\usepackage{amssymb,amsmath,amsfonts,verbatim,cite,subcaption,cleveref}
\usepackage[algo2e,ruled,linesnumbered]{algorithm2e}
\newcommand{\etal}{{et al}. }
\newcommand{\ie}{{i}.{e}., }

\newcommand{\Ttran}{\mathsf{T}}

\newcommand{\de}{\,\mathrm{d}}
\newcommand{\supp}{\mathrm{supp}}
\newcommand{\Tr}{\mathrm{Tr}}
\newcommand{\HE}{\mathcal{H}}
\newcommand{\defi}{\equiv}
\newcommand{\normml}{\left\vert\kern-0.25ex\left\vert\kern-0.25ex\left\vert}
\newcommand{\normmr}{\right\vert\kern-0.25ex\right\vert\kern-0.25ex\right\vert}
\newcommand{\Mv}{\mathbf{v}^{h}}
\newcommand{\MA}{\mathbf{K}^{h}}
\newcommand{\MM}{\mathbf{M}^{h}}
\newcommand{\Mu}{\mathbf{u}^{h}}
\newcommand{\MS}{\mathbf{S}^{h}}
\newcommand{\MH}{\mathbf{H}^{h}}
\newcommand{\MI}{\mathbf{I}^{h}}
\newcommand{\fig}{pdf}

\newcommand{\figsizeD}{0.45\textwidth}

\providecommand{\abs}[1]{\lvert#1\rvert}
\providecommand{\norm}[1]{\lVert#1\rVert}
\providecommand{\dual}[1]{\langle#1\rangle}
\providecommand{\normm}[1]{\normml#1\normmr}
\providecommand{\bigabs}[1]{\bigl\lvert#1\bigr\rvert}
\providecommand{\Bigabs}[1]{\Bigl\lvert#1\Bigr\rvert}

\providecommand{\bignorm}[1]{\bigl\lVert#1\bigr\rVert}
\providecommand{\Bignorm}[1]{\Bigl\lVert#1\Bigr\rVert}
\providecommand{\bigdual}[1]{\bigl\langle#1\bigr\rangle}

\newsiamremark{remark}{Remark}

\begin{document}
    \title{Convergence analysis of the Newton-Schur method for the symmetric elliptic eigenvalue problem \thanks{Version of \today.\funding{W.B. Chen is supported by the National Natural Science Foundation of China (NSFC) 12071090.}}}
    \author{
        Nian~Shao\thanks{School of Mathematical Sciences, Fudan University,  Shanghai, 200433, People's Republic of China ({nshao20@fudan.edu.cn}).}
        \and
        Wenbin~Chen\thanks{Corresponding author. School of Mathematical Sciences and Shanghai Key Laboratory for Contemporary Applied Mathematics, Fudan University,  Shanghai, 200433, People's Republic of China ({wbchen@fudan.edu.cn}).}
    }
    \headers{Convergence analysis of the Newton-Schur method}{Nian Shao and Wenbin Chen }
    \maketitle
    \begin{abstract}
        In this paper, we consider the Newton-Schur method in Hilbert space and obtain quadratic convergence. For the symmetric elliptic eigenvalue problem discretized by the standard finite element method and non-overlapping domain decomposition method, we use the Steklov-Poincar\'e operator to reduce the eigenvalue problem on the domain $\Omega$ into the nonlinear eigenvalue subproblem on $\Gamma$, which is the union of subdomain boundaries. We prove that the convergence rate for the Newton-Schur method is $\epsilon_{N}\leq CH^{2}(1+\ln(H/h))^{2}\epsilon^{2}$, where the constant $C$ is independent of the fine mesh size $h$ and coarse mesh size $H$, and $\epsilon_{N}$ and $\epsilon$ are errors after and before one iteration step respectively. Numerical experiments confirm our theoretical analysis.	
    \end{abstract}
    \section{Introduction}
    It is well-known that the smallest eigenvalue problem is very important in scientific and engineering computations. Suppose $V$ is a Hilbert space with inner product $(\cdot,\cdot)$, the eigenvalue problem can be defined as
    \begin{equation}
        a(v_{\lambda},v)=\lambda\, (v_{\lambda},v)\quad\forall\, v\in V,
        \label{eigeq}
    \end{equation}
    where $a(\cdot,\cdot)$ is a symmetric bilinear form on $V\times V$.

    There are many classical methods for computing the eigenvalue and its corresponding eigenvector in algebraic view \cite{Bai2000,Parlett1998,Saad2011,Wilkinson1965,Golub2013}. However, traditional methods  suffer from slow convergence for problems from fluid dynamics or electronic device simulation \cite{Saad2003}. Therefore, preconditioning techniques are often necessary for converging fast. One of the most famous preconditioned method for eigenvalue problem is the Locally Optimal Block Preconditioned Conjugate Gradient (LOBPCG) method proposed by Knyazev \etal \cite{Knyazev2001,Knyazev2007,Knyazev2003}.

    In PDE view, especially for symmetric elliptic eigenvalue problems, there are many effective methods. For the early important researches on computing the eigenpair, a multigrid method was proposed by Hackbusch in \cite{Hackbusch1979}, a mesh refinement strategy was introduced by McCormick in \cite{McCormick1981} and a multilevel inverse iteration procedure was analyzed by Bank in \cite{Bank1982}. As for the theoretical analysis, the standard Galerkin approximation scheme for computing the approximate eigenpair was analyzed by Babu\v ska and Osborn in \cite{Babuska1987,Babuska1989,Babuska1991}. When it comes to the domain decomposition method, some two-domains decomposition methods for computing the smallest eigenpair were proposed by Lui in \cite{Lui2000} and a Schwarz alternating method for many subdomains case was constructed by Maliassov in \cite{Maliassov1998}. Another effective eigenvalue solver is the two grid method proposed by Xu and Zhou in \cite{Xu2001,Xu2002}. There were some further study about it, such as \cite{Hu2011,Zhou2014,Yang2011}. Moreover, some methods based on correction were also proposed for eigenvalue problem, such as \cite{Lin2015,Xie2019}. Recently, a two-level overlapping hybrid domain decomposition method for solving the large scale elliptic eigenvalue problem by Jacobi-Davidson method was proposed by Wang and Xu in \cite{Wang2018,Wang2019}.

    One important theoretical problem for these methods is to find conditions such that the algorithm is optimal \cite{Toselli2005}, which means that there exists a constant $C$ independent of fine mesh size $h$ such that
    \begin{equation*}
        \epsilon_{N}\leq C\epsilon,
    \end{equation*}
    where $\epsilon_{N}$ and $\epsilon$ are errors after and before one iteration respectively. In early versions of two grid method, some conditions between $h$ and $H$ are needed to ensure the optimality, where $H$ is the mesh size of the coarse space. The first method proposed by Xu and Zhou in \cite{Xu2001,Xu2002} needs $\mathcal{O}(H^{2})=h$, another two level methods based on inverse iteration can be optimal under the condition $\mathcal{O}(H^{4})=h$, see \cite{Hu2011,Yang2011}. Recently, the method proposed by Wang and Xu in \cite{Wang2018,Wang2019} is optimal with no assumptions between $h$ and $H$. For Maxwell eigenvalue problem, similar results can be found in \cite{Liang2022}.

    One popular non-overlapping domain decomposition method for eigenvalue problem is the spectral Schur complements method proposed by Bekas and Saad in \cite{Bekas2005}. It can be regarded as a variation of Automated MultiLevel Substructuring (AMLS) method in \cite{Bennighof2004}, whose numerical implementation can be found in \cite{Gao2008}. Recently, the spectral Schur complement method was developed into the Newton-Schur method in \cite{Kalantzis2018,Kalantzis2020a,Kalantzis2020,Kalantzis2016} by Kalantzis, Li and Saad. All these researches focused on algorithm design and numerical implementation in algebraic view. As for the convergence rate, since the Newton-Schur method is essentially Newton's method, it could be expected to converge quadratically, at least if a sufficiently accurate initial approximation is provided \cite{Kalantzis2020}. But a rigorous theoretical analysis is hard.

    In this paper, we focus on the theoretical analysis of the convergence rate. The Newton-Schur method is studied in the abstract Hilbert space, and the quadratic convergence is obtained under some assumptions on the bilinear form $a(\cdot,\cdot)$ in \eqref{eigeq}. For symmetric elliptic eigenvalue problems discretized by the standard finite element method and non-overlapping domain decomposition method, we use the Steklov-Poincar\'e operator to reduce the eigenvalue problem in the domain $\Omega$ into the nonlinear eigenvalue subproblem on $\Gamma$, which is the union of subdomain boundaries. The assumptions on the bilinear form are verified and the convergence rate of the Newton-Schur method is
    \begin{equation}
        \label{introconfac}
        \epsilon_{N}\leq CH^{2}\bigl(1+\ln(H/h)\bigr)^{2}\epsilon^{2},
    \end{equation}
    where the constant $C$ is independent of $h$ and $H$. The theoretical results are confirmed by our numerical examples for both two-dimensional and three-dimensional elliptic problems. To the best of our knowledge, similar results are not found in the references.

    The outline of this paper is organized as follows: we extend the Newton-Schur method into Hilbert space and provide some results about convergence in section~2. In section~3, we analyze an important problem, the symmetric elliptic eigenvalue problem discretized by the standard finite element method and non-overlapping domain decomposition method, we prove the rate of convergence in \eqref{introconfac}. Finally, numerical experiments are given in section~4. In the rest of this paper, we use notations in \cite{Xu1989}. Let $A \lesssim B$ represent the statement that $A\leq cB$, where the constant $c$ is positive and independent of $h$, $H$ and the variables in $A$ and $B$. The notation $A\gtrsim B$ means $B\lesssim A$ and $A\approx B$ means that $A\lesssim B$ and $B\lesssim A$.

    \section{Newton-Schur method in Hilbert space}
    \subsection{Setting the problem}
    \label{defspace}
    Let $W$ be a Hilbert space with inner product $(\cdot,\cdot)$ and norm $\norm{\cdot}$, and $V\subset W$ be a closed subspace. Suppose $a(\cdot,\cdot)$ is a symmetric bilinear form on $V\times V$. Let $V_{I}$ be a closed subspace of $V$ and suppose $a(\cdot,\cdot)$ is coercive on $V_{I}$, \ie there exists a constant $\alpha>0$ such that
    \begin{equation}
        \label{coercive}
        a(v,v)\geq \alpha\,\norm{v}^{2}
    \end{equation}
    for all $v\in V_{I}$ and with equality for some $v_{I}\in V_{I}$. Since $\alpha>0$, $a(\cdot,\cdot)$ can be regarded as an inner product on $V_{I}$. For a scalar $\rho<\alpha$, let
    \begin{equation}
        \label{defvarrho}
        a_{\rho}(\cdot,\cdot)\defi a(\cdot,\cdot)-\rho\,(\cdot,\cdot),
    \end{equation}
    then $a_{\rho}(\cdot,\cdot)$ is positive definite on $V_{I}$, and the $a_{\rho}$\nobreakdash-orthogonal space of $V_{I}$ is defined as
    \begin{equation}
        \label{defVB}
        V_{B,\rho}\defi\{v\in V\mid a_{\rho}(v,v_{I})=0,\,\forall\, v_{I}\in V_{I}\}.
    \end{equation}
    By using the Lax-Milgram's lemma, we can get the following decomposition for $V$.
    \begin{proposition}
        \label{decompositionV}
        For any $\rho<\alpha$ and all $v\in V$, there exists a unique decomposition
        \begin{equation*}
            v = v_{I}+v_{B},
        \end{equation*}
        where $v_{I}\in V_{I}$ and $v_{B}\in V_{B,\rho}$.
    \end{proposition}

    Suppose $W_{\Gamma}$ is another Hilbert space with one inner product $(\cdot,\cdot)_{\Gamma}$ and the corresponding norm $\norm{\cdot}_{\Gamma}$, and $V_{\Gamma}\subset W_{\Gamma}$ is a closed subspace. Let $\HE_{\rho}$ be a bijective bounded extension from $V_{\Gamma}$ to $V_{B,\rho}$ satisfying
    \begin{equation}
        \label{bound}
        \normm{\HE_{\rho}}\leq c_{\HE}
    \end{equation}
    for all $\rho<\alpha$, where $c_{\HE}$ is a constant independent of $\rho$ and
    \begin{equation*}
        \normm{\HE_{\rho}}\defi \sup_{0\neq u\in V_{\Gamma}}\frac{\norm{\HE_{\rho}u}}{\norm{u}_{\Gamma}}.
    \end{equation*}
    Moreover, for any $\rho_{1},\,\rho_{2}<\alpha$, the extensions $\HE_{\rho_{1}}$ and $\HE_{\rho_{2}}$ satisfy
    \begin{equation}
        \label{ext}
        \HE_{\rho_{1}}\!u-\HE_{\rho_{2}}\!u\in V_{I}
    \end{equation}
    for all $u\in V_{\Gamma}$. The following lemma describes the continuity of $\HE_{\rho}$ respect to $\rho$.
    \begin{lemma}
        \label{continuity}
        Let $\rho_{1},\,\rho_{2}<\alpha$ and $\delta\HE=\HE_{\rho_{1}}\!-\HE_{\rho_{2}}$, then
        \begin{equation*}
            \norm{\delta \HE u}\leq \frac{\abs{\rho_{1}-\rho_{2}}}{\alpha-\rho_{1}}\,\norm{\HE_{\rho_{2}}\!u}\leq\frac{c_{\HE}\abs{\rho_{1}-\rho_{2}}}{\alpha-\rho_{1}}\,\norm{u}_{\Gamma}
        \end{equation*}
        for all $u\in V_{\Gamma}$.
    \end{lemma}
    \begin{proof}
        According to $\delta\HE u=\HE_{\rho_{1}}\!u-\HE_{\rho_{2}}\!u\in V_{I}$ and $V_{I}$ is $a_{\rho_{1}}\!$\nobreakdash-orthogonal to $V_{B,\rho_{1}}$,
        \begin{equation*}
            a_{\rho_{1}}\!(\delta\HE u,\delta\HE u)=a_{\rho_{1}}\!(\delta\HE u,\HE_{\rho_{1}}\!u-\HE_{\rho_{2}}\!u)=-a_{\rho_{1}}\!(\delta\HE u,\HE_{\rho_{2}}\!u).
        \end{equation*}
        Due to $V_{I}$ is $a_{\rho_{2}}\!$\nobreakdash-orthogonal to $V_{B,\rho_{2}}$ and the linearity of $a_{\rho}$ respect to $\rho$,
        \begin{equation*}
            a_{\rho_{1}}\!(\delta\HE u,\HE_{\rho_{2}}\!u)=a_{\rho_{2}}\!(\delta\HE u,\HE_{\rho_{2}}\!u)+(\rho_{2}-\rho_{1})\,(\delta\HE u,\HE_{\rho_{2}}\!u)=(\rho_{2}-\rho_{1})\,(\delta\HE u,\HE_{\rho_{2}}\!u).
        \end{equation*}
        Therefore, by combining these two equations above,
        \begin{equation}
            \label{delta-eps}
            a_{\rho_{1}}\!(\delta\HE u,\delta\HE u)=(\rho_{1}-\rho_{2})\,(\delta\HE u,\HE_{\rho_{2}}\!u).
        \end{equation}
        Since $a_{\rho_{1}}\!(\cdot,\cdot)$ is coercive on $V_{I}$ and $\delta\HE u\in V_{I}$,
        \begin{equation*}
            (\alpha-\rho_{1})\,\norm{\delta\HE u}^{2}\leq (\rho_{1}-\rho_{2})\,(\delta\HE u,\HE_{\rho_{2}}\!u)\leq \abs{\rho_{1}-\rho_{2}}\,\norm{\delta\HE u}\,\norm{\HE_{\rho_{2}}\!u}.
        \end{equation*}
        By eliminating $\norm{\delta\HE u}$ on both sides of the equation above and using the bound for $\normm{\HE_{\rho_{2}}}$,
        \begin{equation*}
            \norm{\delta\HE u}\leq\frac{\abs{\rho_{1}-\rho_{2}}}{\alpha-\rho_{1}}\,\norm{\HE_{\rho_{2}}\!u}\leq \frac{c_{\HE}\,\abs{\rho_{1}-\rho_{2}}}{\alpha-\rho_{1}}\,\norm{u}_{\Gamma}.
        \end{equation*}
        This inequality means $\HE_{\rho_{1}}\!u-\HE_{\rho_{2}}\!u$ goes to zero when $\abs{\rho_{1}-\rho_{2}}\to0$, which leads to the continuity.
    \end{proof}

    The Steklov-Poincar\'e operator $S_{\rho}\colon V_{\Gamma}\mapsto (V_{\Gamma})^{\prime}$ can be defined as
    \begin{equation}
        \label{defsp}
        \dual{S_{\rho}u_{1},u_{2}}\defi a_{\rho}(\HE_{\rho}u_{1},\HE_{\rho}u_{2})
    \end{equation}
    for all $u_{1},\,u_{2}\in V_{\Gamma}$, where $(V_{\Gamma})^{\prime}$ is the dual space of $V_{\Gamma}$ and the bilinear form $\dual{\cdot,\cdot}$ is the duality pairing.
    \subsection{The smallest eigenvalue problem and the Newton-Schur method}
    We are interested in the smallest eigenvalue problem of $a(\cdot,\cdot)$ in $V$, which is to find the smallest $\lambda$ and $\norm{v_{\lambda}}=1$ such that
    \begin{equation}
        \label{eigvar}
        a(v_{\lambda},v)=\lambda\,(v_{\lambda},v)
    \end{equation}
    for all $v\in V$. In the rest of this paper, we assume that the smallest eigenvalue of $a(\cdot,\cdot)$ is simple and $-\infty<\lambda<\alpha$. Actually, due to the variation principle of eigenvalues (see Equation~2.1 in Section~3 of \cite{Weinberger1974}),
    \begin{equation*}
        \lambda = \min_{0\neq v\in V}\frac{a(v,v)}{(v,v)}\leq \min_{0\neq v\in V_{I}}\frac{a(v,v)}{(v,v)}=\alpha,
    \end{equation*}
    with equality only when $v_{\lambda}\in V_{I}$. So $\lambda<\alpha$ means that the eigenvector $v_{\lambda}$ corresponding to the smallest eigenvalue $\lambda$ is not in $V_{I}$. If there exists a scalar $\lambda$ and $u_{\lambda}\in V_{\Gamma}$, $u_{\lambda}\neq0$ such that
    \begin{equation}
        \label{eigsp}
        \dual{S_{\lambda}u_{\lambda},u}=0
    \end{equation}
    for all $u\in V_{\Gamma}$, by using the definition of Steklov-Poincar\'e operator in \cref{defsp},
    \begin{equation*}
        a_{\lambda}(\HE_{\lambda}u_{\lambda},\HE_{\lambda}u)=\dual{S_{\lambda}u_{\lambda},u}=0
    \end{equation*}
    for all $u\in V_{\Gamma}$. On the one hand, as $\HE_{\lambda}$ is injective on $V_{B,\rho}$, $u_{\lambda}\neq0$ leads to $\HE_{\lambda}u_{\lambda}\neq0$. On the other hand, since $\HE_{\lambda}$ is surjective to $V_{B,\lambda}$ and $V_{B,\lambda}$ is the $a_{\lambda}$\nobreakdash-orthogonal complement of $V_{I}$, by \cref{decompositionV},
    \begin{equation}
        \label{eigtmp}
        a_{\lambda}(\HE_{\lambda}u_{\lambda},v)=0
    \end{equation}
    for all $v\in V$, \ie $\lambda$ is the eigenvalue of $a(\cdot,\cdot)$ and $\HE_{\lambda}u_{\lambda}$ is parallel to the corresponding eigenvector. It is easy to verify that if $\lambda$ is the smallest root of \cref{eigsp}, then $\lambda$ is the smallest eigenvalue of $a(\cdot,\cdot)$. So we can use Steklov-Poincar\'e operator to reduce the smallest eigenvalue problem of $a(\cdot,\cdot)$ into the smallest root finding problem of $S_{\rho}$ respect to $\rho$. The Newton-Schur method in \cite{Bekas2005,Kalantzis2018,Kalantzis2020a,Kalantzis2020,Kalantzis2016} is in this framework of algebraic view. In this paper, we take a look at the method in Hilbert space and extend the Newton-Schur method to infinite dimension space. First, let us consider the eigenvalue problem of $S_{\rho}$:
    \begin{equation}
        \label{schureig}
        \dual{S_{\rho}u_{\rho},u}=\theta_{\rho}\,(u_{\rho},u)_{\Gamma}\quad\forall\, u\in V_{\Gamma},
    \end{equation}
    where $(\theta_{\rho},u_{\rho})$ is the smallest eigenpair of $S_{\rho}$ and $\norm{u_{\rho}}_{\Gamma}=1$. The root of \cref{eigsp} is found if we can find a $\lambda$ such that $\theta_{\lambda}=0$. Therefore, the root-finding problem \cref{eigsp} can be transformed into the nonlinear eigenvalue problem \cref{schureig}. In order to apply the Newton-Schur method for \cref{schureig} to find $\rho$ such that $\theta_{\rho}=0$, the first order Fr\'echet derivative $S_{\rho}^{\prime}\defi S^{\prime}(\rho)$ with respect to $\rho$ needs to be calculated first. By $\HE_{\rho}^{\prime}u\in V_{I}$ for all $u\in V_{\Gamma}$, we have the following proposition.
    \begin{proposition}
        \label{defSprime}
        The linear operator $S_{\rho}^{\prime}\colon V_{\Gamma}\to(V_{\Gamma})^{\prime}$ can be expressed as
        \begin{equation*}
            \dual{S_{\rho}^{\prime}u_{1},u_{2}}=-(\HE_{\rho}u_{1},\HE_{\rho}u_{2})\quad\forall\, u_{1},\,u_{2}\in V_{\Gamma}.
        \end{equation*}
    \end{proposition}
    \begin{lemma}
        \label{lemdiff1}
        Assume $(\theta_{\rho},u_{\rho})$ is the smallest eigenpair of $S_{\rho}$ as \cref{schureig}. If $\rho<\alpha$, then the first order derivative $\theta_{\rho}^{\prime}\defi \theta^{\prime}(\rho)$ satisfies
        \begin{equation*}
            \theta_{\rho}^{\prime}=\frac{\dual{S^{\prime}_{\rho}u_{\rho},u_{\rho}}}{(u_{\rho},u_{\rho})_{\Gamma}}=-\frac{(\HE_{\rho}u_{\rho},\HE_{\rho}u_{\rho})}{(u_{\rho},u_{\rho})_{\Gamma}}<0.
        \end{equation*}
    \end{lemma}
    \begin{proof}
    By taking the derivative of \cref{schureig} first:
    \begin{equation}
        \label{diff1}
        \dual{S^{\prime}_{\rho}u_{\rho},u}+\dual{S_{\rho}u^{\prime}_{\rho},u}=\theta_{\rho}^{\prime}\,(u_{\rho},u)_{\Gamma}+\theta_{\rho}\,(u_{\rho}^{\prime},u)_{\Gamma}.
    \end{equation}
    Let $u=u_{\rho}$, since $(\theta_{\rho},u_{\rho})$ is the eigenpair of $S_{\rho}$ and $S_{\rho}$ is symmetric,
    \begin{equation*}
        \dual{S_{\rho}u_{\rho}^{\prime},u_{\rho}}=\dual{S_{\rho}u_{\rho},u_{\rho}^{\prime}}=\theta_{\rho}\,(u_{\rho},u_{\rho}^{\prime})_{\Gamma}=\theta_{\rho}\,(u_{\rho}^{\prime},u_{\rho})_{\Gamma}.
    \end{equation*}
    The lemma is proved by combining these two equations above with \cref{defSprime}.
    \end{proof}

    The Newton iteration $\rho_{N}=\rho-\theta_{\rho}/\theta_{\rho}^{\prime}$ becomes
    \begin{equation*}
        \rho_{N}=\rho-\frac{\dual{S_{\rho}u_{\rho},u_{\rho}}}{\dual{S^{\prime}_{\rho}u_{\rho},u_{\rho}}}
    \end{equation*}
    by \cref{lemdiff1,schureig}. Furthermore,
    let us take $v_{\rho}=\HE_{\rho}u_{\rho}$, since
    \begin{equation*}
        \dual{S_{\rho}u_{\rho},u_{\rho}}=a_{\rho}(\HE_{\rho}u_{\rho},\HE_{\rho}u_{\rho})=a(v_{\rho},v_{\rho})-\rho\,(v_{\rho},v_{\rho}),
    \end{equation*}
    and $\dual{S_{\rho}^{\prime}u_{\rho},u_{\rho}}=-(v_{\rho},v_{\rho})$,  the Newton iteration becomes
    \begin{equation}
        \label{NTstep3}
        \rho_{N}=\frac{a(v_{\rho},v_{\rho})}{(v_{\rho},v_{\rho})}.
    \end{equation}
    Based on the iteration \cref{NTstep3}, the Newton-Schur method in Hilbert space can be sketched as Algorithm~\ref{algo}.\\
    \begin{algorithm2e}[H]
        \label{algo}
        \caption{Newton-Schur method in Hilbert space}
        \KwIn{initial point $\rho$.}
        \KwOut{The approximation for the smallest eigenvalue $\rho$.}
        \Repeat{Convergence}{
            Solve the smallest eigenvalue problem \cref{schureig} in $V_{\Gamma}$ to get $u_{\rho}$\;
            Extend $u_{\rho}$ to $V$ by $v_{\rho} = \HE_{\rho}u_{\rho}$\;
            Update $\rho$ with \cref{NTstep3}:
            \begin{equation*}
                \rho = \frac{a(v_{\rho},v_{\rho})}{(v_{\rho},v_{\rho})};
            \end{equation*}
        }
    \end{algorithm2e}
    Since each iteration point $\rho$ is a Rayleigh quotient of $a(\cdot,\cdot)$ and $\lambda$ is the smallest eigenvalue of $a(\cdot,\cdot)$, $\rho\geq\lambda$ always holds during the algorithm. Let $\epsilon$ and $\epsilon_{N}$ be errors before and after one step of Newton iteration respectively:
    \begin{equation}
        \label{deferr}
        \epsilon = \rho-\lambda\quad\text{and}\quad
        \epsilon_{N} = \rho_{N}-\lambda.
    \end{equation}
    The following proposition can guarantee the convergence of Algorithm~\ref{algo}.
    \begin{proposition}
        \label{convergence}
        Let $\rho_{0}\in(\lambda,\alpha)$ be the initial point of Algorithm~\ref{algo}, where $\alpha$ is the constant for coercivity and $\lambda$ is the smallest eigenvalue of $a(\cdot,\cdot)$. The Newton-Schur method is convergent and at each iteration
        \begin{equation*}
            0\leq\epsilon_{N}<\epsilon.
        \end{equation*}
    \end{proposition}
    \begin{proof}
        According to \cref{NTstep3,deferr},
        \begin{equation*}
            \epsilon_{N} = \epsilon-\frac{\theta_{\rho}}{\theta_{\rho}^{\prime}}.
        \end{equation*}
        Due to \cref{lemdiff1}, we know $\theta_{\rho}<\theta_{\lambda}=0$. By combining it with $\theta_{\rho}^{\prime}<0$,
        \begin{equation*}
            \epsilon_{N} = \epsilon-\frac{\theta_{\rho}}{\theta_{\rho}^{\prime}} < \epsilon.
        \end{equation*}
        On the other hand, as $\rho_{N}$ is a Rayleigh quotient of $a(\cdot,\cdot)$, we have $\epsilon_{N} = \rho_{N}-\lambda\geq0$.
    \end{proof}

    \subsection{Convergence factor for the Newton-Schur method}
    In order to analyze the convergence, we need to use the well-known result about Newton's method.
    \begin{proposition}
        \label{quadcon}
        Let $\epsilon$ and $\epsilon_{N}$ be errors before and after one step of Newton iteration. Suppose the iterative point $\rho$ is in $(\lambda,\alpha)$, the error $\epsilon_{N}$ satisfies
        \begin{equation*}
            \epsilon_{N}= \frac{\theta_{\xi}^{\prime\prime}}{2\theta_{\rho}^{\prime}}\,\epsilon^{2},
        \end{equation*}
        where $\lambda\leq \xi\leq\rho$ and $\theta_{\xi}^{\prime\prime}\defi \theta^{\prime\prime}(\xi)$.
    \end{proposition}

    The Newton-Schur method will converge quadratically if $\theta_{\xi}^{\prime\prime}/2\theta_{\rho}^{\prime}$ is bounded. In this paper, we refer to the upper bound of $\theta_{\xi}^{\prime\prime}/2\theta_{\rho}^{\prime}$ as convergence factor $\eta$, \ie
    \begin{equation*}
        \eta = \sup_{\lambda\leq\xi\leq\rho\leq\rho_{0}}\frac{\theta_{\xi}^{\prime\prime}}{2\theta_{\rho}^{\prime}},
    \end{equation*}
    then errors satisfy $\epsilon_{N}\leq \eta\epsilon^{2}$.
    \begin{lemma}
        \label{lemdiff2}
        Assume $(\theta_{\rho},u_{\rho})$ is the smallest eigenpair of $S_{\rho}$ as \cref{schureig}. If $\rho<\alpha$, then the second order derivative $\theta_{\rho}^{\prime\prime}$ respect to $\rho$ satisfies
        \begin{equation*}
            \theta_{\rho}^{\prime\prime}\leq0.
        \end{equation*}
    \end{lemma}
    \begin{proof}
    First, let us take the derivative in \cref{diff1} respect to $\rho$ again,
    \begin{equation*}
        \dual{S^{\prime\prime}_{\rho}u_{\rho},u}+2\,\dual{S_{\rho}^{\prime}u^{\prime}_{\rho},u}+\dual{S_{\rho}u^{\prime\prime}_{\rho},u}=\theta_{\rho}^{\prime\prime}\,(u_{\rho},u)_{\Gamma}+2\,\theta_{\rho}^{\prime}\,(u_{\rho}^{\prime},u)_{\Gamma}+\theta_{\rho}\,(u_{\rho}^{\prime\prime},u)_{\Gamma}.
    \end{equation*}
    Let $u=u_{\rho}$ and by using \cref{schureig}, it becomes
    \begin{equation*}
        \theta_{\rho}^{\prime\prime}(u_{\rho},u_{\rho})_{\Gamma}=
        \dual{S^{\prime\prime}_{\rho}u_{\rho},u_{\rho}}-2\,\theta_{\rho}^{\prime}\,(u_{\rho}^{\prime},u_{\rho})_{\Gamma}+2\,\dual{S_{\rho}^{\prime}u^{\prime}_{\rho},u_{\rho}}.
    \end{equation*}
    Let $u=u_{\rho}^{\prime}$ in \cref{diff1},
    \begin{equation}
        \label{diff12}
        \dual{S^{\prime}_{\rho}u_{\rho},u_{\rho}^{\prime}}+\dual{S_{\rho}u^{\prime}_{\rho},u_{\rho}^{\prime}}=\theta_{\rho}^{\prime}\,(u_{\rho},u_{\rho}^{\prime})_{\Gamma}+\theta_{\rho}\,(u_{\rho}^{\prime},u_{\rho}^{\prime})_{\Gamma}.
    \end{equation}
    Combining these two equations above, we know
    \begin{equation}
        \label{diff2}
        \theta_{\rho}^{\prime\prime}(u_{\rho},u_{\rho})_{\Gamma}=
        \dual{S^{\prime\prime}_{\rho}u_{\rho},u_{\rho}}+2\,\theta_{\rho}\,(u_{\rho}^{\prime},u_{\rho}^{\prime})_{\Gamma}-2\,\dual{S_{\rho}u^{\prime}_{\rho},u_{\rho}^{\prime}}.
    \end{equation}
    Since $\theta_{\rho}$ is the smallest eigenvalue of $S_{\rho}$, we have $\theta_{\rho}\,(u_{\rho}^{\prime},u_{\rho}^{\prime})_{\Gamma}\leq\dual{S_{\rho}u^{\prime}_{\rho},u_{\rho}^{\prime}}$. Thus, the $\theta_{\rho}^{\prime\prime}$ can be bounded by
    \begin{equation}
        \label{diff2cb1}
        \theta_{\rho}^{\prime\prime}\,(u_{\rho},u_{\rho})_{\Gamma}\leq\dual{S^{\prime\prime}_{\rho}u_{\rho},u_{\rho}}.
    \end{equation}
    By taking the derivative of $S_{\rho}^{\prime}$ in \cref{lemdiff1},
    \begin{equation*}
        \dual{S^{\prime\prime}_{\rho}u_{\rho},u_{\rho}}
        =\lim_{\varepsilon\to0}\frac{1}{\varepsilon}\Big((\HE_{\rho}u_{\rho},\HE_{\rho}u_{\rho})-(\HE_{\rho+\varepsilon}u_{\rho},\HE_{\rho+\varepsilon}u_{\rho})\Big).
    \end{equation*}
    Let $\delta\HE  = \HE_{\rho+\varepsilon}-\HE_{\rho}$, then
    \begin{equation}
        \label{spp}
        \begin{aligned}
            \dual{S^{\prime\prime}_{\rho}u_{\rho},u_{\rho}}
            &=\lim_{\varepsilon\to0}\frac{1}{\varepsilon}\Big((\HE_{\rho}u_{\rho}-\HE_{\rho+\varepsilon}u_{\rho},\HE_{\rho}u_{\rho})+(\HE_{\rho}u_{\rho}-\HE_{\rho+\varepsilon}u_{\rho},\HE_{\rho+\varepsilon}u_{\rho})\Big)\\
            &=\lim_{\varepsilon\to0}\frac{1}{\varepsilon}\Big(-(\delta\HE u_{\rho},\HE_{\rho}u_{\rho})-(\delta\HE u_{\rho},\HE_{\rho+\varepsilon}u_{\rho})\Big)\\
            &=-\lim_{\varepsilon\to0}\frac{1}{\varepsilon}\Big(2\,(\delta\HE u_{\rho},\HE_{\rho+\varepsilon}u_{\rho})-(\delta\HE u_{\rho},\delta\HE u_{\rho})\Big).
        \end{aligned}
    \end{equation}
    According to \cref{delta-eps},
    \begin{equation*}
        a_{\rho}(\delta\HE u_{\rho},\delta\HE u_{\rho})
        =\varepsilon (\delta\HE u_{\rho},\HE_{\rho+\varepsilon}u_{\rho}).
    \end{equation*}
    Substituting it into \cref{spp}, we have
    \begin{equation}
        \label{diff2cb2}
        \dual{S^{\prime\prime}_{\rho}u_{\rho},u_{\rho}}=-\lim_{\varepsilon\to0}\frac{1}{\varepsilon}\Big(\frac{2}{\varepsilon}\,a_{\rho}(\delta\HE u_{\rho},\delta\HE u_{\rho})-(\delta\HE u_{\rho},\delta\HE u_{\rho})\Big).
    \end{equation}
    Since $a_{\rho}(\cdot,\cdot)$ is coercive on $V_{I}$, then
    \begin{equation}
        \label{diff2cb3}
        (\delta\HE u_{\rho},\delta\HE u_{\rho})\leq \frac{1}{\alpha-\rho}\,a_{\rho}(\delta\HE u_{\rho},\delta\HE u_{\rho}).
    \end{equation}
    Therefore, by combining \cref{diff2cb2,diff2cb1,diff2cb3},
    \begin{equation}
        \label{thetadiff2}
        \theta_{\rho}^{\prime\prime}\leq
        \frac{\dual{S^{\prime\prime}_{\rho}u_{\rho},u_{\rho}}}{(u_{\rho},u_{\rho})_{\Gamma}}\leq
        \frac{-1}{(u_{\rho},u_{\rho})_{\Gamma}}\lim_{\varepsilon\to0}\bigg(\Big(\frac{2}{\varepsilon^{2}}-\frac{1}{\varepsilon\,(\alpha-\rho)}\Big)\,a_{\rho}(\delta\HE u_{\rho},\delta\HE u_{\rho})\bigg)\leq0.
    \end{equation}
    The lemma is proved.
    \end{proof}
    \begin{remark}
        \cref{lemdiff2} also holds for other eigenvalues of $S_{\rho}$ if $\rho<\alpha$.
    \end{remark}

    \cref{quadcon} can be relaxed to
    \begin{equation*}
        \epsilon_{N}\leq\sup_{\xi\in(\lambda,\rho)}\frac{\theta_{\xi}^{\prime\prime}}{2\theta_{\xi}^{\prime}}\,\epsilon^{2}
    \end{equation*}
    based on \cref{lemdiff2}.
    \begin{theorem}
        \label{main1}
        Suppose $\rho_{0}\in(\lambda,\alpha)$ is the initial point of Algorithm~\ref{algo}, denote $g_{\rho}$ the gap between the smallest two eigenvalues of $S_{\rho}$, if there exists a constant $c_{g}>0$ such that
        \begin{equation}
            \label{gap}
            g_{\rho}\geq c_{g}
        \end{equation}
        for all $\rho\in(\lambda,\rho_{0})$, then  Algorithm~\ref{algo} is quadratic convergence, and
        \begin{equation*}
            \epsilon_{N}\leq \Big(\frac{1}{\alpha-\rho_{0}}+\frac{c_{\HE}^{2}}{c_{g}}\Big)\,\epsilon^{2}.
        \end{equation*}
    \end{theorem}
    \begin{proof}
        Suppose $\norm{u_{\rho}}_{\Gamma}=1$, according to \cref{diff2}, $(\theta_{\rho}^{\prime\prime}/2\theta_{\rho}^{\prime})$ can be divided into two parts:
        \begin{equation}
            \label{est1}
            -\dual{S_{\rho}^{\prime\prime}u_{\rho},u_{\rho}}\leq C_{1}\,\norm{v_{\rho}}^{2},
        \end{equation}
        and
        \begin{equation}
            \label{est2}
            \dual{S_{\rho}u_{\rho}^{\prime},u_{\rho}^{\prime}}-\theta_{\rho}(u_{\rho}^{\prime},u_{\rho}^{\prime})_{\Gamma}\leq C_{2}\,\norm{v_{\rho}}^{2},
        \end{equation}
        where $v_{\rho}=\HE_{\rho}u_{\rho}$. According to \cref{spp},
        \begin{equation*}
            \begin{aligned}
            -\dual{S_{\rho}^{\prime\prime}u_{\rho},u_{\rho}}&=\lim_{\varepsilon\to0}\frac{1}{\varepsilon}\Big(2\,(\delta \HE u_{\rho},\HE_{\rho+\varepsilon} u_{\rho})-(\delta \HE u_{\rho},\delta \HE u_{\rho})\Big)\\
            &=\lim_{\varepsilon\to0}\frac{1}{\varepsilon}\Big(2\,(\delta \HE u_{\rho},\HE_{\rho} u_{\rho})+(\delta \HE u_{\rho},\delta \HE u_{\rho})\Big)\\
            &\leq \lim_{\varepsilon\to0}\frac{\norm{\delta\HE u_{\rho}}}{\abs{\varepsilon}}\,(2\,\norm{\HE_{\rho}u_{\rho}} +\norm{\delta \HE u_{\rho}})
            \end{aligned}
        \end{equation*}
        Due to \cref{continuity},
        \begin{equation*}
            \begin{aligned}
                -\dual{S_{\rho}^{\prime\prime}u_{\rho},u_{\rho}}&\leq
            \lim_{\varepsilon\to0}\frac{\norm{\HE_{\rho}u_{\rho}}}{(\alpha-\rho-\varepsilon)}\Big(2\,\norm{\HE_{\rho} u_{\rho}} +\frac{\abs{\varepsilon}}{\alpha-\rho}\norm{\HE_{\rho} u_{\rho}}\Big)\\ &=\frac{2}{\alpha-\rho}\norm{v_{\rho}}^{2}\leq \frac{2}{\alpha-\rho_{0}}\norm{v_{\rho}}^{2}.
            \end{aligned}
        \end{equation*}
        Then the constant $C_{1}$ in estimation \cref{est1} can be $2\,(\alpha-\rho_{0})^{-1}$. For the estimation \cref{est2}, an orthogonality about $u_{\rho}$ and $u_{\rho}^{\prime}$ is needed. By taking the derivative on both sides of $\norm{u_{\rho}}_{\Gamma}^{2}=1$, we have $(u_{\rho},u_{\rho}^{\prime})_{\Gamma}=0$, which means $u_{\rho}$ and $u_{\rho}^{\prime}$ are orthogonal in $(\cdot,\cdot)_{\Gamma}$. Let
        \begin{equation}
            \label{defVGr}
            V_{\Gamma,\rho}\defi \{u\in V_{\Gamma}\mid (u,u_{\rho})_{\Gamma}=0\}.
        \end{equation}
        Since $g_{\rho}$ is the gap between $\theta_{\rho}$ and the smallest eigenvalue of $S_{\rho}$ in $V_{\Gamma,\rho}$, and $u_{\rho}^{\prime}\in V_{\Gamma,\rho}$,
        \begin{equation}
            \label{est20}
            \dual{S_{\rho}u_{\rho}^{\prime},u_{\rho}^{\prime}}- \theta_{\rho}\,(u_{\rho}^{\prime},u_{\rho}^{\prime})_{\Gamma}\geq g_{\rho}\,\norm{u_{\rho}^{\prime}}_{\Gamma}^{2}.
        \end{equation}
        Let $u=u_{\rho}^{\prime}$ in \cref{diff1}, by combining the orthogonality of $u_{\rho}$ and $u_{\rho}^{\prime}$,
        \begin{equation}
            \label{est21}
            \dual{S_{\rho}u_{\rho}^{\prime},u_{\rho}^{\prime}}- \theta_{\rho}\,(u_{\rho}^{\prime},u_{\rho}^{\prime})_{\Gamma}=
            -\dual{S_{\rho}^{\prime}u_{\rho},u_{\rho}^{\prime}}+\theta_{\rho}^{\prime}\,(u_{\rho},u_{\rho}^{\prime})=
            -\dual{S_{\rho}^{\prime}u_{\rho},u_{\rho}^{\prime}}.
        \end{equation}
        According to the Cauchy-Schwarz inequality and \cref{defSprime},
        \begin{equation}
            \label{est22}
            -\dual{S_{\rho}^{\prime}u_{\rho},u_{\rho}^{\prime}}=
            (\HE_{\rho}u_{\rho},\HE_{\rho}u_{\rho}^{\prime})\leq \norm{v_{\rho}}\,\norm{\HE_{\rho}u_{\rho}^{\prime}}\leq c_{\HE}\norm{v_{\rho}}\,\norm{u_{\rho}^{\prime}}_{\Gamma}.
        \end{equation}
        By combining the inequalities \cref{est20,est21,est22} and eliminating $\norm{u_{\rho}^{\prime}}_{\Gamma}$ on both sides,
        \begin{equation}
            \label{est23}
            \norm{u_{\rho}^{\prime}}_{\Gamma}\leq \frac{c_{\HE}}{g_{\rho}}\,\norm{v_{\rho}} .
        \end{equation}
        And combining the inequalities \cref{est21,est22,est23}, we have
        \begin{equation*}
            \dual{S_{\rho}u_{\rho}^{\prime},u_{\rho}^{\prime}}- \theta_{\rho}\,(u_{\rho}^{\prime},u_{\rho}^{\prime})_{\Gamma}\leq
            \frac{c_{\HE}^{2}}{g_{\rho}}\,\norm{v_{\rho}}^{2}\leq
            \frac{c_{\HE}^{2}}{c_{g}}\,\norm{v_{\rho}}^{2},
        \end{equation*}
        which means $C_{2}$ in estimation \cref{est2} can be set as $c_{\HE}^{2}c_{g}^{-1}$. Therefore, the constant in \cref{main1} can be $\big((\alpha-\rho_{0})^{-1}+c_{\HE}^{2}c_{g}^{-1}\big)$,
        which is independent of $\rho$.
    \end{proof}

    The condition \cref{gap} in \cref{main1} may be difficult to verified directly, here we give a lemma for the existence of $c_{g}$.

    \begin{lemma}
        \label{lemgap}
        Let $\widehat{\lambda}$ be the second smallest eigenvalue of $a(\cdot,\cdot)$. If the initial point $\rho_{0}\in(\lambda,\widehat{\lambda})$ and $\widehat{\lambda}<\alpha$, then the constant $c_{g}$ in \cref{main1} exists and
        \begin{equation*}
            c_{g}\geq \frac{\widehat{\lambda}-\rho_{0}}{\widehat{\lambda}-\lambda}\,\theta^{(2)}_{\lambda},
        \end{equation*}
        where $\theta^{(2)}_{\lambda}$ is the second smallest eigenvalue of $S_{\lambda}$.
    \end{lemma}
    \begin{proof}
        Suppose $\theta_{\rho}^{(2)}$ is the second smallest eigenvalue of $S_{\rho}$, since $\theta_{\rho}^{(2)}$ is concave,
        \begin{equation*}
            \frac{\theta^{(2)}_{\rho}-\theta^{(2)}_{\lambda}}{\rho-\lambda}\geq\frac{\theta^{(2)}_{\rho}-\theta^{(2)}_{\widehat{\lambda}}}{\rho-\widehat{\lambda}}
        \end{equation*}
        for $\lambda<\rho<\widehat{\lambda}$. Since $\widehat{\lambda}$ is the second smallest eigenvalue of $a(\cdot,\cdot)$, we know that $\theta_{\widehat{\lambda}}^{(2)}=0$. By using $\theta^{(2)}_{\lambda}>\theta^{(2)}_{\widehat{\lambda}}=0$ and
        \begin{equation*}
            \theta^{(2)}_{\rho}\geq\frac{\widehat{\lambda}-\rho}{\widehat{\lambda}-\lambda}\,\theta^{(2)}_{\lambda}\geq\frac{\widehat{\lambda}-\rho_{0}}{\widehat{\lambda}-\lambda}\,\theta^{(2)}_{\lambda},
        \end{equation*}
        the proof is finished.
    \end{proof}

    If the bilinear form $a(\cdot,\cdot)$ satisfies some more conditions, the convergence factor $\eta$ can be estimated more specifically.
    \begin{lemma}
        \label{lemgapmulti}
        Suppose $\rho_{0}\in (\lambda,\widehat{\lambda})$ is the initial point, where $\lambda$ and $\widehat{\lambda}$ are the smallest two eigenvalues of $a(\cdot,\cdot)$ respectively. If there exists a constant $c_{t}>0$ such that
        \begin{equation}
            \label{hiltrace}
            a(\HE_{\lambda}u,\HE_{\lambda}u)\geq c_{t}\norm{u}_{\Gamma}^{2}
        \end{equation}
        holds for all $u\in V_{\Gamma}$, then there is a lower bound for the constant $c_{g}$, \ie
        \begin{equation*}
            c_{g}\geq \frac{\widehat{\lambda}-\rho_{0}}{\widehat{\lambda}}\,c_{s}c_{t},
        \end{equation*}
        where $0<c_{s}=\theta_{\lambda}^{(2)}/\theta_{\lambda}^{(3)} \leq 1$
        and $\theta_{\lambda}^{(2)}$, $\theta_{\lambda}^{(3)}$ are the second and third smallest eigenvalue of $S_{\lambda}$ respectively, including multiplicities.
    \end{lemma}
    \begin{proof}
        Let $\theta_{\lambda}$, $\theta^{(2)}_{\lambda}$ and $\theta^{(3)}_{\lambda}$ be the smallest three eigenvalues of $S_{\lambda}$ with corresponding eigenvectors $u_{\lambda}$, $\widehat{u}_{2}$ and $\widehat{u}_{3}$, where $\norm{u_{\lambda}}_{\Gamma}=\norm{\widehat{u}_{2}}_{\Gamma}=\norm{\widehat{u}_{3}}_{\Gamma}=1$ and $(u_{\lambda},\widehat{u}_{2})_{\Gamma}=(\widehat{u}_{2},\widehat{u}_{3})_{\Gamma}=(\widehat{u}_{3},u_{\lambda})_{\Gamma}=0$. Let
        \begin{equation*}
            \widehat{u}_{\lambda}=\left\{
                \begin{array}{cc}
                    \dfrac{(\HE_{\lambda}u_{\lambda},\HE_{\lambda}\widehat{u}_{2})\,\widehat{u}_{3}
                    -(\HE_{\lambda}u_{\lambda},\HE_{\lambda}\widehat{u}_{3})\,\widehat{u}_{2}}
                    {\norm{(\HE_{\lambda}u_{\lambda},\HE_{\lambda}\widehat{u}_{2})\,\widehat{u}_{3}
                    -(\HE_{\lambda}u_{\lambda},\HE_{\lambda}\widehat{u}_{3})\,\widehat{u}_{2}}_{\Gamma}} & \text{if } (\HE_{\lambda}u_{\lambda},\HE_{\lambda}\widehat{u}_{2}) \neq 0,\\ \\
                    \widehat{u}_{2} & \text{if } (\HE_{\lambda}u_{\lambda},\HE_{\lambda}\widehat{u}_{2}) = 0,
                \end{array}\right.
        \end{equation*}
        then it is easy to be verified that $\norm{\widehat{u}_{\lambda}}_{\Gamma}=1$ and $(\HE_{\lambda}u_{\lambda},\HE_{\lambda}\widehat{u}_{\lambda})=0$. Since $\HE_{\lambda}u_{\lambda}$ is the eigenvector corresponding to $\lambda$ of $a(\cdot,\cdot)$, $(\HE_{\lambda}u_{\lambda},\HE_{\lambda}\widehat{u}_{\lambda})=0$ leads to $\HE_{\lambda}\widehat{u}_{\lambda}$ lies in the eigenspace corresponding to the eigenvalue at least $\widehat{\lambda}$, then
        \begin{equation*}
            \begin{aligned}
                a_{\lambda}(\HE_{\lambda}\widehat{u}_{\lambda},\HE_{\lambda}\widehat{u}_{\lambda})&=a(\HE_{\lambda}\widehat{u}_{\lambda},\HE_{\lambda}\widehat{u}_{\lambda})-\lambda\,(\HE_{\lambda}\widehat{u}_{\lambda},\HE_{\lambda}\widehat{u}_{\lambda})\\ &\geq \Big(1-\frac{\lambda}{\widehat{\lambda}}\Big)\,a(\HE_{\lambda}\widehat{u}_{\lambda},\HE_{\lambda}\widehat{u}_{\lambda})\geq c_{t}\Big(1-\frac{\lambda}{\widehat{\lambda}}\Big).
            \end{aligned}
        \end{equation*}
        Since $\widehat{u}_{\lambda}$ is the linear combination of $\widehat{u}_{2}$ and $\widehat{u}_{3}$, and $\norm{\widehat{u}_{\lambda}}_{\Gamma}=1$,
        \begin{equation*}
            \theta_{\lambda}^{(2)}\leq \dual{S_{\lambda}\widehat{u}_{\lambda},\widehat{u}_{\lambda}}\leq \theta_{\lambda}^{(3)}.
        \end{equation*}
        Therefore, the second smallest eigenvalue of $S_{\lambda}$ satisfies
        \begin{equation*}
            \theta^{(2)}_{\lambda}=\frac{\theta_{\lambda}^{(2)}}{\theta_{\lambda}^{(3)}}\,\,\theta_{\lambda}^{(3)}
            \geq c_{s}\,\dual{S_{\lambda}\widehat{u}_{\lambda},\widehat{u}_{\lambda}}=
            c_{s}\,a_{\lambda}(\HE_{\lambda}\widehat{u}_{\lambda},\HE_{\lambda}\widehat{u}_{\lambda})\geq
            (1-\frac{\lambda}{\widehat{\lambda}})\,c_{s}c_{t}.
        \end{equation*}
        Then the constant $c_{g}$ can be bounded by
        \begin{equation*}
            c_{g}\geq \frac{\widehat{\lambda}-\rho_{0}}{\widehat{\lambda}}\,c_{s}c_{t}
        \end{equation*}
        by using \cref{lemgap}.
    \end{proof}
    \begin{corollary}
        \label{cor2}
        Let $\lambda$ and $\widehat{\lambda}$ be the smallest two eigenvalues of $a(\cdot,\cdot)$ satisfying $\lambda<\widehat{\lambda}<\alpha$. Suppose the constants $c_{s}$ and $c_{t}$ are defined as \cref{lemgapmulti}, if the initial point $\rho_{0}\in(\lambda,\widehat{\lambda})$, then the rate of convergence for Algorithm~\ref{algo} is
        \begin{equation*}
            \epsilon_{N}\leq \Big(\frac{1}{\alpha-\rho_{0}}+\frac{c_{\HE}^{2}\,\widehat{\lambda}}{c_{s}c_{t}\,(\widehat{\lambda}-\rho_{0})}\Big)\,\epsilon^{2}.
        \end{equation*}
    \end{corollary}

    \section{Finite element method for the symmetric elliptic eigenvalue problem}
    \label{femsec}
    In this section, we  focus on the smallest eigenvalue problem of the symmetric elliptic operator. The problem is discretized by finite element method and solved by the Newton-Schur method, the space $V_{I}$, $V_{B,\rho}$ and $V_{\Gamma}$ are constructed by non-overlapping domain decomposition method.
    \subsection{Elliptic eigenvalue problem and finite element method}
    Let $\Omega\subset\mathbb{R}^{d}$, where $d=2$ or $3$ be a bounded convex polygonal domain, the smallest eigenvalue problem of the symmetric elliptic operator is to find the smallest $\lambda\in\mathbb{R}$ and sufficiently smooth $v_{\lambda}$ such that
    \begin{equation}
        \label{lap}
        \begin{aligned}
            Av_{\lambda}&=\lambda v_{\lambda}\quad\text{in }\Omega,\\
            v_{\lambda}&=0     \qquad\text{on }\partial\Omega,
        \end{aligned}
    \end{equation}
    where
    \begin{equation*}
        Av\defi -\sum_{i,j=1}^{d}\frac{\partial}{\partial x_{i}}\Big(a_{ij}(x)\frac{\partial v(x)}{\partial x_{j}}\Big)
    \end{equation*}
    and $\int_{\Omega}v_{\lambda}^{2}\de x=1$. Assume the matrix $\{a_{ij}(x)\}_{i,\,j=1}^{d}$ is symmetric and uniformly positive definite  and $a_{ij}(x)\in C^{0,1}(\overline{\Omega})$ for $i,j=1,\dotsc,d$. Let $W=L^{2}(\Omega)$, $(\cdot,\cdot)$ be $L^{2}$ inner product on $\Omega$ and $V$ be the Sobolev space $H_{0}^{1}(\Omega)$, then the variational form of \cref{lap} is
    \begin{equation}
        \label{lapvar}
        a(v_{\lambda},v)\defi\sum_{i,j=1}^{d}\int_{\Omega}a_{ij}\frac{\partial v_{\lambda}}{\partial x_{i}}\frac{\partial v}{\partial x_{j}}\de x=\lambda\,(v_{\lambda},v)\defi\lambda\int_{\Omega}v_{\lambda} v\de x\quad \forall\, v\in V.
    \end{equation}
    Let $\lambda$ be the smallest eigenvalue of \cref{lapvar}, it is well-known that $\lambda$ is simple (see Theorem~2 in Section~6.5 of \cite{Evans2010}). Moreover, we assume that $a(\cdot,\cdot)$ is equivalent to the square of the $H_{0}^{1}$ norm, \ie
    \begin{equation}
        \label{eqh1norm}
        \abs{v}_{a}\defi \bigl(a(v,v)\bigr)^{1/2}\approx \Bigl(\int_{\Omega}\abs{\nabla v}^{2}\de x\Bigr)^{1/2}.
    \end{equation}
    We construct continuous and piecewise linear element spaces $V^{H}\subset V^{h}\subset V$ based on quasi-uniform triangular partitions $\mathcal{T}^{H}$ and $\mathcal{T}^{h}$, where $0<h<H<1$ are mesh sizes of $\mathcal{T}^{H}$ and $\mathcal{T}^{h}$ respectively and $\mathcal{T}^{h}$ is refined by $\mathcal{T}^{H}$. By using the finite element discretization, the variational form of \cref{lapvar} becomes
    \begin{equation}
        \label{lapvarh}
        a(v_{\lambda}^{h},v^{h})=\lambda^{h}\,(v_{\lambda}^{h},v^{h})\quad\forall\, v^{h}\in V^{h},
    \end{equation}
    where $v_{\lambda}^{h}\in V^{h}$. Then the discrete elliptic operator $A^{h}$ is defined as
    \begin{equation*}
        (A^{h}v_{1}^{h},v_{2}^{h})\defi a(v_{1}^{h},v_{2}^{h})\quad\forall\, v_{1}^{h},\,v_{2}^{h}\in V^{h}.
    \end{equation*}
    The convergence of the discrete eigenvalues  can be found in \cite{Babuska1987}.
    \begin{proposition}
        \label{lapeigest}
        Suppose $A^{h}$ and $A^{H}$ are discrete elliptic operators with mesh sizes $h$ and $H$ respectively, then, following properties hold.
        \begin{equation}
                \lambda^{h}\!-\lambda\approx h^{2},\quad
                \widehat{\lambda}^{h}\!-\widehat{\lambda}\approx h^{2}\quad\text{and}\quad\lambda^{H}-\lambda\approx H^{2},
        \end{equation}
        where $\lambda$, $\lambda^{h}$ and $\lambda^{H}$ are the smallest eigenvalues of $A$, $A^{h}$ and $A^{H}$, $\widehat{\lambda}$ and $\widehat{\lambda}^{h}$ are the second smallest eigenvalues of $A$ and $A^{h}$ respectively.
    \end{proposition}

    From \cref{lapeigest} and the variational principle of eigenvalues (see Equation~2.1 in Section~3 of \cite{Weinberger1974}), the gap between $\lambda^{h}$ and $\lambda^{H}$ can be bounded by
    \begin{equation}
        \label{gaphH}
        0\leq\lambda^{H}-\lambda^{h}\lesssim H^{2}.
    \end{equation}
    In the rest of this paper, we take the initial point $\rho_{0}=\lambda^{H}$, thus $0\leq\rho_{0}-\lambda^{h}\lesssim H^{2}$.

    \subsection{Non-overlapping domain decomposition methods}
    Suppose $\Omega$ is divided into $N$ non-overlapping convex polygonal subdomains $\{\Omega_{k}\}_{k=1}^{N}$ with diameter no more than $H$ and the union of the boundaries are denoted by $\Gamma=\cup_{k=1}^{N}\partial\Omega_{k}$. Let $W_{\Gamma}=L^{2}(\Gamma)$, $(\cdot,\cdot)_{\Gamma}$ be an inner product on $\Gamma$, whose corresponding norm $\norm{\cdot}_{\Gamma}$ is spectral equivalent to $\norm{\cdot}_{L^{2}(\Gamma)}$, and $V_{\Gamma}=H_{*}^{1/2}(\Gamma)$, where $u\in H_{*}^{1/2}(\Gamma)$ means that the restriction of $u$ on $\partial\Omega_{k}$ belongs to $H^{1/2}(\Omega_{k})$ for all $k=1,\dotsc,N$. Let
    \begin{equation*}
        \norm{u^{h}}_{H_{*}^{1/2}(\Gamma)}\defi \Bigl(\sum_{k=1}^{N}\norm{u^{h}}_{H^{1/2}(\partial\Omega_{k})}^{2}\Bigr)^{1/2},
    \end{equation*}
    where the scaled full-norm (see, for examples, \cite{Toselli2005,Xu1998}) in $\Omega_{k}$ is defined as
    \begin{equation*}
        \begin{aligned}
            \norm{u^{h}}_{H^{1}(\Omega_{k})}&\defi \bigl(\abs{u^{h}}_{H^{1}(\Omega_{k})}^{2}+H^{-2}\norm{u^{h}}_{L^{2}(\Omega_{k})}^{2}\bigr)^{1/2},\\
            \norm{u^{h}}_{H^{1/2}(\partial\Omega_{k})}&\defi \bigl(\abs{u^{h}}_{H^{1/2}(\partial\Omega_{k})}^{2}+H^{-1}\norm{u^{h}}_{L^{2}(\partial\Omega_{k})}^{2}\bigr)^{1/2}.
        \end{aligned}
    \end{equation*}
    Then we can define $V_{\Gamma}^{h}$ as the trace space of $V^{h}$ on $\Gamma$ and $V_{I}^{h}$ as the subspace of $V^{h}$ whose members vanish at $\Gamma$, \ie
    \begin{equation*}
        \begin{aligned}
            V_{\Gamma}^{h}&\defi \{u^{h}\in H_{*}^{1/2}(\Gamma)\mid \exists\, v^{h}\in V^{h},\,u^{h}=\Tr(v^{h})\},\\
            V_{I}^{h}& \defi \{v^{h}\in V^{h}\mid \Tr(v^{h})=0\},
        \end{aligned}
    \end{equation*}
    where $\Tr\colon H^{1}(\Omega)\mapsto H_{*}^{1/2}(\Gamma)$ is the trace map defined as $\Tr(v)=v$. In order to define the extension, a lemma for the coercivity of $a(\cdot,\cdot)$ in $V_{I}^{h}$ is needed.
    \begin{lemma}
        \label{lapcoer}
        Suppose $\alpha$ is the smallest eigenvalue of $a(\cdot,\cdot)$ in $V_{I}^{h}$, \ie
        \begin{equation*}
            \alpha = \min_{v^{h}\in V_{I}^{h}}\frac{a(v^{h},v^{h})}{(v^{h},v^{h})},
        \end{equation*}
        then we have $\alpha\gtrsim H^{-2}$.
    \end{lemma}
    \begin{proof}
        For every $v^{h}\in V_{I}^{h}$, it can be decomposed as $v^{h}=\sum_{k=1}^{N}v_{k}^{h}$, where $\supp(v_{k})\subset \overline{\Omega}_{k}$. Since $a(v_{i}^{h},v_{j}^{h})=(v_{i}^{h},v_{j}^{h})=0$ for all $i\neq j$, we know
        \begin{equation*}
            \frac{a(v^{h},v^{h})}{(v^{h},v^{h})}=\frac{\sum_{k=1}^{N}a(v_{k}^{h},v_{k}^{h})}{\sum_{k=1}^{N}(v_{k}^{h},v_{k}^{h})}\geq \min_{1\leq k\leq N}\frac{a(v_{k}^{h},v_{k}^{h})}{(v_{k}^{h},v_{k}^{h})}\geq \min_{1\leq k\leq N}\lambda_{k}^{h},
        \end{equation*}
        where $\lambda_{k}^{h}$ refers to the smallest eigenvalue of $a(\cdot,\cdot)$ restricted on $\Omega_{k}$. Since the diameter of $\Omega_{k}$ is no more than $H$, then $\alpha\gtrsim H^{-2}$ holds due to the Poincar\'e inequality.
    \end{proof}

    For all $\rho<\alpha$, denote $a_{\rho}(\cdot,\cdot)\defi a(\cdot,\cdot)-\rho\,(\cdot,\cdot)$, the $a_{\rho}$\nobreakdash-orthogonal space of $V_{I}^{h}$ can be defined as
    \begin{equation*}
        V_{B,\rho}^{h}\defi\{v^{h}\in V^{h}\mid a_{\rho}(v^{h},v_{I}^{h})=0,\, \forall\, v_{I}^{h}\in V_{I}^{h}\}.
    \end{equation*}
    Then the discrete $a_{\rho}$\nobreakdash-harmonic extension $\HE_{\rho}^{h}\colon V_{\Gamma}^{h}\mapsto V_{B,\rho}^{h}$ can be defined. For all $u^{h}\in V_{\Gamma}^{h}$, let $\HE_{\rho}^{h}u^{h}$ be the solution of
    \begin{equation}
        \label{defHhuh}
        \begin{aligned}
            a_{\rho}(\HE_{\rho}^{h}u^{h},v_{I}^{h})&=0\quad\,\,\, \forall\, v_{I}^{h}\in V_{I}^{h},\\
            \HE_{\rho}^{h}u^{h}&=u^{h}\quad\text{on }\Gamma.
        \end{aligned}
    \end{equation}
    Here are some propositions about the extension $\HE_{\rho}^{h}$ (see \cite{Brenner1999}).
    \begin{proposition}
        \label{bij}
        The extension $\HE_{\rho}^{h}$ is bijective since the functions in $V_{B,\rho}^{h}$ are completely determined by their values on $\Gamma$.
    \end{proposition}
    \begin{proposition}
        \label{minrq}
        For any $u^{h}\in V^{h}_{\Gamma}$ and $v^{h}\in V^{h}$, if $u^{h}$ is the trace of $v^{h}$ on $\Gamma$, \ie $\Tr(v^{h})=u^{h}$, then $a_{\rho}(\HE_{\rho}^{h}u^{h},\HE_{\rho}^{h}u^{h})\leq a_{\rho}(v^{h},v^{h})$ holds for all $\rho<\alpha$.
    \end{proposition}
    \begin{proposition}
        \label{minevalS}
        For any $u^{h}\in V^{h}_{\Gamma}$ and $\lambda^{h}\leq\rho\leq\rho_{0}$, $H\norm{u^{h}}_{\Gamma}^{2}\lesssim\abs{\HE_{\rho}^{h}u^{h}}_{a}^{2}$.
    \end{proposition}
    \begin{lemma}
        \label{refdomain}
        Suppose $D\subset \mathbb{R}^{d}$, where $d=2$ or $3$, is a convex polygonal domain with unit diameter. Let
            \begin{equation*}
                \mathcal{A}(v_{1},v_{2}) = \sum_{i,j=1}^{d}\int_{D}a_{ij}(x)\frac{\partial v_{1}}{\partial x_{i}}\frac{\partial v_{2}}{\partial x_{j}}-c(x)\,v_{1}v_{2}\,\de x
            \end{equation*}
        be a symmetric positive definite bilinear form on $H^{1}(D)$ satisfying
        \begin{equation}
            \label{h1eqtmp}
            \mathcal{A}(v,v)\approx \norm{v}_{H_{1}(D)}^{2},
        \end{equation}
        where $a_{ij}(x)$ and $c(x)\in C^{0,1}(\overline{D})$ for $i,j=1,\dotsc,d$. For all $u\in H^{1/2}(\partial D)$, let $\HE u$ be the solution of
        \begin{equation}
            \label{defHutmp}
            \begin{aligned}
                \mathcal{A}(\HE u,v) &= 0\quad\forall\, v\in H_{0}^{1}(D),\\
                \HE u&=u\quad\text{on }\partial D,
            \end{aligned}
        \end{equation}
        then the following estimation holds:
        \begin{equation*}
            \norm{\HE u}_{L^{2}(D)}\lesssim\norm{u}_{H^{-1/2}(\partial D)}.
        \end{equation*}
    \end{lemma}
    \begin{proof}
        For any $v\in L^{2}(D)$, let $\psi\in H_{0}^{1}(D)$ be the solution of
        \begin{equation*}
            \mathcal{A}(w,\psi)-\dual{\frac{\partial\psi}{\partial\nu},w}_{\partial D}=(v,w)_{L^{2}(D)}\quad \forall\, w\in H^{1}(D),
        \end{equation*}
        where
        \begin{equation*}
            \frac{\partial\psi}{\partial\nu}=\sum_{i,j=1}^{d}a_{ij}(x)\cos(\mathbf{n},x_{i})\,\frac{\partial\psi}{\partial x_{j}}
        \end{equation*}
        with $\mathbf{n}$ is the outer normal to the boundary $\partial D$ and
        \begin{equation*}
            \dual{\frac{\partial\psi}{\partial\nu}, w}_{\partial D}=\int_{\partial D}\frac{\partial\psi}{\partial\nu}w\de s.
        \end{equation*}
        By using Aubin-Nitsche's trick, for all $u\in H^{1/2}(\partial D)$
        \begin{equation*}
            \norm{\HE u}_{L^{2}(D)}=\sup_{0\neq v\in L^{2}(D)}\frac{(\HE u,v)_{L^{2}(D)}}{\norm{v}_{L^{2}(D)}}=\sup_{0\neq v\in L^{2}(D)}\frac{-\dual{\HE u,\frac{\partial\psi}{\partial\nu}}_{\partial D}+\mathcal{A}(\HE u,\psi)}{\norm{v}_{L^{2}(D)}}.
        \end{equation*}
        Since $\psi\in H_{0}^{1}(D)$ and $\HE u$ is the solution of \cref{defHutmp}, $\mathcal{A}(\HE u,\psi)=0$, then
        \begin{equation}
            \label{an1}
            \norm{\HE u}_{L^{2}(D)}=\sup_{0\neq v\in L^{2}(D)}\frac{-\dual{\HE u,\frac{\partial\psi}{\partial\nu}}_{\partial D}}{\norm{v}_{L^{2}(D)}}.
        \end{equation}
        By using the Cauchy-Schwarz inequality and trace theorem,
        \begin{equation}
            \label{an2}
            \Bigabs{\dual{\HE u,\frac{\partial\psi}{\partial\nu}}_{\partial D}}
            \leq \norm{u}_{H^{-1/2}(\partial D)}\Bignorm{\frac{\partial\psi}{\partial\nu}}_{H^{1/2}(\partial D)}\lesssim  \norm{u}_{H^{-1/2}(\partial D)}\norm{\psi}_{H^{2}(D)}.
        \end{equation}
        Due to the $H^{2}$ regularity (see Theorem~3.2.1.2 in \cite{Grisvard2011}),
        \begin{equation}
            \label{an3}
            \norm{\psi}_{H^{2}(D)}\lesssim \norm{v}_{L^{2}(D)}.
        \end{equation}
        By combining \cref{an1,an2,an3}, the proof is finished.
    \end{proof}
    \begin{corollary}
        If the diameter of domain $D$ is $H$, it can be verified
        that
        \begin{equation*}
            \norm{\HE u}_{L^{2}(D)}\lesssim \norm{u}_{H^{-1/2}(\partial D)}\lesssim H^{1/2}\norm{u}_{L^{2}(\partial D)}
        \end{equation*}
        by a scaling argument and the Sobolev inequality.
    \end{corollary}

    The following lemma gives a finite element method version of \cref{refdomain}.
    \begin{lemma}
        Suppose $D$ and $\mathcal{A}$ are defined as \cref{refdomain}. Let $V_{D}^{h}\subset H^{1}(D)$ be the continuous and piecewise linear elements space based on the quasi-uniform triangular partition $\mathcal{T}_{D}$ with mesh size $h$ and $V_{\partial D}^{h}$ is the trace space of $V_{D}^{h}$ on $\partial D$. For all $u^{h}\in V_{\partial D}^{h}$, let $\HE^{h}u^{h}$ be the solution of
        \begin{equation}\
            \label{defHhutmp}
            \begin{aligned}
                \mathcal{A}(\HE^{h}u^{h},v^{h})&=0\,\,\,\quad\forall\, v^{h}\in H_{0}^{1}(D)\cap V_{D}^{h},\\
                \HE^{h}u^{h}&=u^{h}\quad\text{on }\partial D,
            \end{aligned}
        \end{equation}
        then the following estimation holds:
        \begin{equation*}
            \norm{\HE^{h}u^{h}}_{L^{2}(D)}\lesssim \norm{u^{h}}_{H^{-1/2}(\partial D)}.
        \end{equation*}
    \end{lemma}
    \begin{proof}
        According to \cref{refdomain}, it is enough to prove that
        \begin{equation*}
            \norm{\HE^{h}u^{h}\!-\HE u^{h}}_{L^{2}(D)}\lesssim\norm{u^{h}}_{H^{-1/2}(\partial D)}.
        \end{equation*}
        By the $L^{2}$ estimation for $\HE^{h}u^{h}$ (see Theorem~5.4.8 in \cite{Brenner2008}),
        \begin{equation}
            \label{est41}
            \norm{\HE^{h}u^{h}\!-\HE u^{h}}_{L^{2}(D)}\lesssim h\,\norm{\HE^{h}u^{h}\!-\HE u^{h}}_{H^{1}(D)}.
        \end{equation}
        Since $\HE^{h}u^{h}\!-\HE u^{h}$ vanishes on $\partial D$ and $\HE u^{h}$ is the solution of \cref{defHhutmp},
        \begin{equation*}
            \mathcal{A}(\HE u^{h},\HE^{h}u^{h}\!-\HE u^{h})=0.
        \end{equation*}
        By using the norm equivalence \cref{h1eqtmp} and the Cauchy-Schwarz inequality,
        \begin{equation*}
            \begin{aligned}
                \norm{\HE^{h}u^{h}\!-\HE u^{h}}_{H^{1}(D)}^{2}&\approx \mathcal{A}(\HE^{h}u^{h}\!-\HE u^{h},\HE^{h}u^{h}\!-\HE u^{h})\\
                &=\mathcal{A}(\HE^{h}u^{h}\!-\HE u^{h},\HE^{h}u^{h})\\
                &\lesssim\norm{\HE^{h}u^{h}\!-\HE u^{h}}_{H^{1}(D)}\norm{\HE^{h}u^{h}}_{H^{1}(D)}.
            \end{aligned}
        \end{equation*}
        Then eliminating $\norm{\HE^{h}u^{h}\!-\HE u^{h}}_{H^{1}(D)}$ on both sides, we have
        \begin{equation}
            \label{est42}
            \norm{\HE^{h}u^{h}\!-\HE u^{h}}_{H^{1}(D)}\lesssim \norm{\HE^{h}u^{h}}_{H^{1}(D)}.
        \end{equation}
        Since $\HE^{h}u^{h}$ is the solution of \cref{defHhutmp}, by using \cref{minrq} and the extension theorem (see Lemma~3.77 in \cite{Mathew2008}),
        \begin{equation}
            \label{est43}
            \norm{\HE^{h}u^{h}}_{H^{1}(D)}\lesssim \norm{u^{h}}_{H^{1/2}(\partial D)}.
        \end{equation}
        And by combining \cref{est41,est42,est43},
        \begin{equation*}
            \norm{\HE^{h}u^{h}\!-\HE u^{h}}_{L^{2}(D)}\lesssim h\,\norm{u^{h}}_{H^{1/2}(\partial D)}.
        \end{equation*}
        Since $u^{h}$ is piecewise linear on $\partial D$, by using the inverse estimation (see Theorem~4.1 and Theorem~4.6 in \cite{Dahmen2004}),
        \begin{equation*}
            \norm{\HE^{h}u^{h}\!-\HE u^{h}}_{L^{2}(D)}\lesssim h\,\norm{u^{h}}_{H^{1/2}(\partial D)}\lesssim \norm{u^{h}}_{H^{-1/2}(\partial D)},
        \end{equation*}
        which finishes the proof.
    \end{proof}
    \begin{corollary}
        \label{H1/2}
        If the diameter of domain $D$ is $H$, it can be verified
        that
        \begin{equation*}
            \norm{\HE^{h} u^{h}}_{L^{2}(D)}\lesssim \norm{u^{h}}_{H^{-1/2}(\partial D)}\lesssim H^{1/2}\norm{u^{h}}_{L^{2}(\partial D)}
        \end{equation*}
        by a scaling argument and the Sobolev inequality.
    \end{corollary}

    Suppose $\HE_{\rho,k}^{h}u^{h}$ is the solution of
    \begin{equation*}
        \begin{aligned}
            a_{\rho}(\HE_{\rho,k}^{h}u^{h},v_{i,k}^{h})&=0\,\,\,\quad\forall\, v_{i,k}^{h}\in V_{I}^{h},\,\supp(v_{i,k}^{h})\subset \overline{\Omega}_{k},\\
            \HE_{\rho,k}^{h}u^{h}&=u^{h}\quad\text{on }\partial\Omega_{k},
        \end{aligned}
    \end{equation*}
    then $\HE_{\rho,k}^{h}u^{h}$ is well-defined in $\Omega_{k}$ and $\HE_{\rho}^{h}u^{h}=\HE_{\rho,k}^{h}u^{h}$ in $\Omega_{k}$ for all $1\leq k\leq N$, where $\HE_{\rho}^{h}u^{h}$ is the solution of \cref{defHhuh}. In other words, the extension $\HE_{\rho}^{h}u^{h}$ can be computed in each subdomain by the boundary value problem separately.
    \begin{lemma}
        \label{lapext}
        Assume $\lambda\leq\rho\leq\rho_{0}<\alpha$, the extension operator $\HE_{\rho}$ is bounded with norm
        \begin{equation*}
            \normm{\HE_{\rho}^{h}}\defi \sup_{0\neq u^{h}\in V_{\Gamma}^{h}}\frac{\norm{\HE_{\rho}^{h}u^{h}} }{\norm{u^{h}}_{\Gamma}},
        \end{equation*}
        and the bound for $\normm{\HE_{\rho}^{h}}$ is $c_{\HE}$, where $c_{\HE}\lesssim H^{1/2}$.
    \end{lemma}
    \begin{proof}
        For all $u^{h}\in V_{\Gamma}^{h}$, according to \cref{H1/2},
        \begin{equation*}
            \norm{\HE_{\rho}^{h}u^{h}}^{2}=\int_{\Omega}\abs{\HE_{\rho}^{h}u^{h}}^{2}\de x=\sum_{k=1}^{N}\int_{\Omega_{k}}\abs{\HE_{\rho}^{h}u^{h}}^{2}\de x =\sum_{k=1}^{N}\int_{\Omega_{k}}\abs{\HE_{\rho,k}^{h}u^{h}}^{2}\de x\lesssim H\,\norm{u^{h}}_{\Gamma}^{2}.
        \end{equation*}
        Therefore  $c_{\HE}=\max\limits_{\lambda\leq\rho\leq\rho_{0}}\normm{\HE_{\rho}^{h}}\lesssim H^{1/2}$.
    \end{proof}

    The Steklov-Poincar\'e operator $S_{\rho}^{h}\colon V_{\Gamma}^{h}\mapsto(V_{\Gamma}^{h})^{\prime}$ can be defined as
    \begin{equation*}
        \dual{S_{\rho}^{h}u_{1}^{h},u_{2}^{h}} \defi a_{\rho}(\HE_{\rho}^{h}u_{1}^{h},\HE_{\rho}^{h}u_{2}^{h}),
    \end{equation*}
    and the eigenvalue problem on $\Gamma$ is
    \begin{equation}
        \label{feschureig}
        \dual{S_{\rho}^{h}u_{\rho}^{h},u^{h}}=\theta_{\rho}^{h}\,(u_{\rho}^{h},u^{h})_{\Gamma}\quad\forall\, u^{h}\in V_{\Gamma}^{h}.
    \end{equation}
    Now  we can use Algorithm~\ref{algo} to calculate the elliptic eigenvalue problem \cref{lap}.
    \begin{remark}
        Suppose $\{\phi_{j}\}_{j=1}^{n}$ is the basis of $V_{\Gamma}^{h}$, then the discrete $L^{2}$ inner product on $\Gamma$ can be defined as
        \begin{equation*}
            (u,u^{*})_{l^{2}(\Gamma)}\defi h^{d-1}\sum_{j=1}^{n}u_{j}u_{j}^{*}
        \end{equation*}
        for $u=\sum_{j=1}^{n}u_{j}\phi_{j}$ and $u^{*}=\sum_{j=1}^{n}u_{j}^{*}\phi_{j}$. Since $\norm{u}_{L^{2}(\Gamma)}\approx \norm{u}_{l^{2}(\Gamma)}$ (see Equation~2.2 in \cite{Bramble1991}), the convergence analysis also holds for the discrete $L^{2}$ norm. Moreover, the mass matrix for \cref{feschureig} becomes a scalar matrix, which makes it easy to compute.
    \end{remark}
    \begin{lemma}
        \label{eqevec}
        For all $\lambda^{h}\leq\rho\leq\rho_{0}$, $\norm{\HE_{\rho}^{h}u_{\rho}^{h}}\approx H^{1/2}\norm{u_{\rho}^{h}}_{\Gamma}$.
    \end{lemma}
    \begin{proof}
        On the one hand, by \cref{lapext}, we have $\norm{\HE_{\rho}^{h}u_{\rho}^{h}}\lesssim H^{1/2}\norm{u_{\rho}^{h}}_{\Gamma}$. On the other hand, since $u_{\rho}^{h}$ is the eigenvector of \cref{feschureig} and $\theta_{\rho}^{h}\leq0$,
        \begin{equation*}
            a_{\rho}(\HE_{\rho}^{h}u_{\rho}^{h},\HE_{\rho}^{h}u_{\rho}^{h})=\dual{S_{\rho}^{h}u_{\rho}^{h},u_{\rho}^{h}}=\theta_{\rho}^{h}\norm{u_{\rho}^{h}}_{\Gamma}^{2}\leq0.
        \end{equation*}
        The lemma is obtained by using \cref{minevalS}.
    \end{proof}
    \subsection{Convergence analysis}
    In order to analyze the convergence factor of Algorithm~\ref{algo}, some estimations on the eigenvalue of $S_{\rho}^{h}$ should be calculated first. Similar to \cref{lapext}, a bounded bijective extension $\HE_{\rho}$ from $V_{\Gamma}$ to $V_{B,\rho}$ can be defined as
    \begin{equation}
        \label{defHuh}
        \begin{aligned}
            a_{\rho}(\HE_{\rho}u,v_{I})&=0\quad \forall\, v_{I}\in V_{I},\\
            \HE_{\rho}u&=u\quad\text{on }\Gamma,
        \end{aligned}
    \end{equation}
    where $V_{\Gamma}$ is the trace space of $V$ on $\Gamma$, and $V_{I}$ is the subspace of $V$ whose members vanish at $\Gamma$ and $V_{B,\rho}$ is the $a_{\rho}$\nobreakdash-orthogonal space of $V_{I}$, \ie
    \begin{equation*}
        \begin{aligned}
            V_{I} &\defi \{v\in V\mid \Tr(v)=0\},\\
            V_{B,\rho}&\defi \{v\in V\mid a_{\rho}(v,v_{I})=0,\,\forall\, v_{I}\in V_{I}\}.
        \end{aligned}
    \end{equation*}

    The Steklov-Poincar\'e operator $S_{\rho}\colon V_{\Gamma}\mapsto (V_{\Gamma})^{\prime}$ can be defined as
    \begin{equation*}
        \dual{S_{\rho}u_{1},u_{2}}\defi a_{\rho}(\HE_{\rho}u_{1},\HE_{\rho}u_{2}),
    \end{equation*}
    and its eigenvalue problem is
    \begin{equation*}
        \dual{S_{\rho}u_{\rho},u} = \theta_{\rho}\,(u_{\rho},u)_{\Gamma}\quad\forall\, u\in V_{\Gamma}.
    \end{equation*}
    Now  we can generalize Theorem~3.1 in \cite{Babuska1987} to estimate the gap between eigenvalues of $S_{\rho}$ and $S_{\rho}^{h}$.
    \begin{lemma}
        \label{lemeigschurest}
        Suppose $\theta_{1}^{h}\leq\theta_{2}^{h}\leq\theta_{3}^{h}$ and $\theta_{1}<\theta_{2}\leq\theta_{3}$ are the smallest three eigenvalues of $S_{\rho}^{h}$ and $S_{\rho}$ respectively, where $\lambda^{h}\leq\rho\leq\rho_{0}$, then
        \begin{equation*}
            \lim_{h\to0}\abs{\theta_{j}^{h}\!-\theta_{j}}=0\quad j=1,2,3.
        \end{equation*}
    \end{lemma}
    \begin{proof}
        Since $\theta_{2}$ may be equal to $\theta_{3}$, we need to consider both cases. According to the well-known variational principles of eigenvalue (see Theorem~2.3 in Section~3 of \cite{Weinberger1974}), for $j=1,2,3$,
        \begin{equation}
            \label{deftheta}
            \theta_{j}=\!\!\!\!\min_{\begin{subarray}{c}
                U_{j}\subset V_{\Gamma}\\\dim U_{j}=j
            \end{subarray}}\!\!\!
            \max_{u\in U_{j}}\frac{\dual{S_{\rho}u,u}}{(u,u)_{\Gamma}}
        \quad\text{and}\quad
            \theta_{j}^{h}=\!\!\!\!\min_{\begin{subarray}{c}
                U_{j}^{h}\subset V_{\Gamma}^{h}\\\dim U_{j}^{h}=j
            \end{subarray}}\!\!\!
            \max_{u^{h}\in U_{j}^{h}}\frac{\dual{S_{\rho}^{h}u^{h},u^{h}}}{(u^{h},u^{h})_{\Gamma}}.
        \end{equation}
        Suppose the first minimal of $\theta_{j}$ and $\theta_{j}^{h}$ in \cref{deftheta} are reached by $U_{j}$ and $U_{j}^{h}$ respectively. Let
        \begin{equation}
            \label{defthetabar}
            \bar{\theta}_{j}^{h}\defi \max_{u^{h}\in U_{j}^{h}}\frac{\dual{S_{\rho}u^{h},u^{h}}}{(u^{h},u^{h})_{\Gamma}},
        \end{equation}
        then the error $\abs{\theta_{j}-\theta_{j}^{h}}$ can be decomposed as
        \begin{equation}
            \begin{aligned}
                \label{eiggapbbsk}
                &\abs{\theta_{j}-\theta_{j}^{h}}\leq \abs{\theta_{j}-\bar{\theta}_{j}^{h}}+\abs{\bar{\theta}_{j}^{h}\!-\theta_{j}^{h}}\\
                \leq&\Bigabs{
                    \max_{u\in U_{j}}\frac{\dual{S_{\rho}u,u}}{(u,u)_{\Gamma}}
                    -\max_{u^{h}\in U_{j}^{h}}\frac{\dual{S_{\rho}u^{h},u^{h}}}{(u^{h},u^{h})_{\Gamma}}
                    }+
                \max_{u^{h}\in U_{j}^{h}}\Bigabs{
                    \frac{\dual{S_{\rho}u^{h},u^{h}}-\dual{S_{\rho}^{h}u^{h},u^{h}}}{(u^{h},u^{h})_{\Gamma}}
                    }.
            \end{aligned}
        \end{equation}
        For the first term, it can be regarded as the error between the eigenvalue of $S_{\rho}$ in $V_{\Gamma}$ and $V_{\Gamma}^{h}$. Let
        \begin{equation*}
            \dual{S_{\rho,\beta}u_{1},u_{2}} \defi \dual{S_{\rho}u_{1},u_{2}} + \beta\,(u_{1},u_{2})_{\Gamma},
        \end{equation*}
        it is obvious that eigenvalues of $S_{\rho,\beta}$ are eigenvalues of $S_{\rho}$ plus $\beta$ in both $V_{\Gamma}$ and $V_{\Gamma}^{h}$, so it is enough to analyze the error between the eigenvalue of $S_{\rho,\beta}$ in $V_{\Gamma}$ and $V_{\Gamma}^{h}$. When $\beta=\rho_{0}\,c_{\HE}^{2}+H^{-1}$, by the trace theorem and the norm equivalence \cref{eqh1norm},
        \begin{equation*}
            \begin{aligned}
                \norm{u}_{H^{1/2}_{*}(\Gamma)}^{2}&
                \lesssim  \sum_{k=1}^{N}\abs{u}_{H^{1/2}(\partial\Omega_{k})}^{2}
                +\frac{1}{H}\norm{u}_{\Gamma}^{2}
                \lesssim \abs{\HE_{\rho}u}_{a}^{2}+(\beta-\rho\,\normm{\HE_{\rho}}^{2})\,\norm{u}_{\Gamma}^{2}
                \\ &\leq \abs{\HE_{\rho}u}_{a}^{2}+\beta\,\norm{u}_{\Gamma}^{2}=\dual{S_{\rho,\beta}u,u}.
            \end{aligned}
        \end{equation*}
        On the other hand, by the definition of $a_{\rho}(\cdot,\cdot)$, \cref{minrq,eqh1norm},
        \begin{equation*}
            \begin{aligned}
                \dual{S_{\rho,\beta}u,u}&=a_{\rho}(\HE_{\rho}u,\HE_{\rho}u)+\beta\,(u,u)_{\Gamma}\leq a_{\rho}(\HE u,\HE u)+\beta\,\norm{u}_{\Gamma}^{2}\\
                &\leq\,  \abs{\HE u}_{a}^{2}+\beta\,\norm{u}_{\Gamma}^{2}\approx \sum_{k=1}^{N}\abs{\HE u}_{H^{1}(\Omega_{k})}^{2}+\beta\,\norm{u}_{\Gamma}^{2}\\
                &\lesssim \sum_{k=1}^{N}\abs{u}_{H^{1/2}(\partial\Omega_{k})}^{2}+\beta\,\norm{u}_{\Gamma}^{2}
                \lesssim \norm{u}_{H^{1/2}_{*}(\Gamma)}^{2},
            \end{aligned}
        \end{equation*}
        where $\HE u$ is an extension satisfying $\abs{\HE u}_{H^{1}(\Omega_{k})}\lesssim \abs{u}_{H^{1/2}(\partial\Omega_{k})}$ for all $1\leq k\leq N$ (see Lemma~3.78 in \cite{Mathew2008}). So $S_{\rho,\beta}$ is bounded in $H_{*}^{1/2}(\Gamma)$. Let
        \begin{equation*}
            \norm{u}_{\beta}\defi \big(\dual{S_{\rho,\beta}u,u}\big)^{1/2}\approx \norm{u}_{H^{1/2}_{*}(\Gamma)},
        \end{equation*}
        by using the Theorem~3.1 in \cite{Babuska1987},
        \begin{equation}
            \label{esttheta}
            \Bigabs{\Big(
            \max_{u\in U_{j}}\frac{\dual{S_{\rho}u,u}}{(u,u)_{\Gamma}}
            -\max_{u^{h}\in U_{j}^{h}}\frac{\dual{S_{\rho}u^{h},u^{h}}}{(u^{h},u^{h})_{\Gamma}}
            \Big)}\lesssim \,\delta_{j,h}^{2},
        \end{equation}
        where
        \begin{equation*}
            \delta_{j,h}=\inf_{\begin{subarray}{c}
                u\in M(\theta_{j})\\ \norm{u}_{\beta}=1
            \end{subarray}}
            \inf_{u^{h}\in V_{\Gamma}^{h}}\norm{u-u^{h}}_{\beta}\quad j=1,2,
        \end{equation*}
        and
        \begin{equation*}
            \delta_{3,h}=
            \left\{
                \begin{array}{ll}
                    \inf\limits_{\begin{subarray}{c}
                        u\in M(\theta_{3})\\ \norm{u}_{\beta}=1
                    \end{subarray}}
                    \inf\limits_{u^{h}\in V_{\Gamma}^{h}}\norm{u-u^{h}}_{\beta}\quad\text{if }\theta_{2}\neq\theta_{3},\\
                    \inf\limits_{\begin{subarray}{c}
                        u\in M(\theta_{2})\\ \norm{u}_{\beta}=1\\ (u,u_{2})_{\beta}=0
                    \end{subarray}}
                    \!\!\inf\limits_{u^{h}\in V_{\Gamma}^{h}}\norm{u-u^{h}}_{\beta}\quad\text{if }\theta_{2}=\theta_{3},
                \end{array}\right.
        \end{equation*}
        where $u_{2}$ is the choice of the first infimum of $\delta_{2,h}$ and $M(\theta_{j})$ is the eigenspace corresponding to $\theta_{j}$. Due to Theorem~3.2.3 of \cite{Ciarlet2002}, we have $\lim_{h\to0}\,\delta_{j,h}=0$, thus, the first term in \cref{eiggapbbsk} goes to zeros as $h\to0$, \ie
        \begin{equation}
            \label{eiggapbbsk1}
            \lim_{h\to0}\abs{\theta_{j}-\bar{\theta}_{j}^{h}}=0.
        \end{equation}
        For the second term,
        \begin{equation*}
                \dual{S_{\rho}u^{h},u^{h}}-\dual{S_{\rho}^{h}u^{h},u^{h}}=
                a_{\rho}(\HE_{\rho}u^{h}\!-\HE_{\rho}^{h}u^{h},\HE_{\rho}u^{h}\!+\HE_{\rho}^{h}u^{h}).
        \end{equation*}
        Since $\HE_{\rho}u^{h}\!-\HE_{\rho}^{h}u^{h}$ vanishes on $\Gamma$, due to $\HE_{\rho}u^{h}$ is the solution of \cref{defHuh},
        \begin{equation*}
            a_{\rho}(\HE_{\rho}u^{h}\!-\HE_{\rho}^{h}u^{h},\HE_{\rho}u^{h}) = 0.
        \end{equation*}
        Since $a_{\rho}(\cdot,\cdot)$ is positive definite on $V_{I}$, combining these two equations above, we have
        \begin{equation}
            \label{est31}
            \begin{aligned}
                \Bigabs{\,\dual{S_{\rho}u^{h},u^{h}}-\dual{S_{\rho}^{h}u^{h},u^{h}}\,}
                &=\Bigabs{\,a_{\rho}(\HE_{\rho}u^{h}\!-\HE_{\rho}^{h}u^{h},\HE_{\rho}u^{h}\!+\HE_{\rho}^{h}u^{h})\,}\\
                &=\Bigabs{\,a_{\rho}(\HE_{\rho}u^{h}\!-\HE_{\rho}^{h}u^{h},-\HE_{\rho}u^{h}\!+\HE_{\rho}^{h}u^{h})\,}\\
                &=a_{\rho}(\HE_{\rho}u^{h}\!-\HE_{\rho}^{h}u^{h},\HE_{\rho}u^{h}\!-\HE_{\rho}^{h}u^{h}).
            \end{aligned}
        \end{equation}
        Similar to C\'ea's lemma, for all $v^{h}\in V_{B,\rho}^{h}$ with trace $u^{h}$ on $\Gamma$,
        \begin{equation}
            \label{est32}
            a_{\rho}(\HE_{\rho}u^{h}\!-\HE_{\rho}^{h}u^{h},\HE_{\rho}u^{h}\!-\HE_{\rho}^{h}u^{h})
            \leq a_{\rho}(\HE_{\rho}u^{h}\!-v^{h},\HE_{\rho}u^{h}\!-v^{h})\lesssim \abs{\HE_{\rho}u^{h}\!-v^{h}}_{H^{1}}^{2}
        \end{equation}
        due to the norm equivalence \cref{eqh1norm} and the Poincar\'e inequality. Combining \cref{eiggapbbsk,est31,est32} we know
        \begin{equation}
            \label{est33}
            \begin{aligned}
                \abs{\bar{\theta}_{j}^{h}\!-\theta_{j}^{h}}&\lesssim
                \max_{u^{h}\in U_{j}^{h}}\inf_{v^{h}\in V_{B,\rho}^{h}}\frac{\abs{\HE_{\rho}u^{h}\!-v^{h}}_{H^{1}}^{2}}{\norm{u^{h}}_{\Gamma}^{2}}\\
                &=\max_{u^{h}\in U_{j}^{h}}\Big\{\frac{\abs{\HE_{\rho}u^{h}}_{H^{1}}^{2}}{\norm{u^{h}}_{\Gamma}^{2}}\inf_{v^{h}\in V_{B,\rho}^{h}}\frac{\abs{\HE_{\rho}u^{h}\!-v^{h}}_{H^{1}}^{2}}{\abs{\HE_{\rho}u^{h}}_{H^{1}}^{2}}\Big\}.
            \end{aligned}
        \end{equation}
        According to the definition of $\bar{\theta}_{j}^{h}$ \cref{defthetabar} and the norm equivalence \cref{eqh1norm},
        \begin{equation*}
            \bar{\theta}_{j}^{h}\geq \frac{\dual{S_{\rho}u^{h},u^{h}}}{(u^{h},u^{h})_{\Gamma}}=\frac{a_{\rho}(\HE_{\rho}u^{h},\HE_{\rho}u^{h})}{(u^{h},u^{h})_{\Gamma}}\approx \frac{\abs{\HE_{\rho}u^{h}}_{H^{1}}^{2}-\rho \norm{\HE_{\rho}u^{h}}^{2}}{\norm{u^{h}}_{\Gamma}^{2}}
        \end{equation*}
        for all $u^{h}\in U_{j}^{h}$. So
        \begin{equation}
            \label{est34}
            \frac{\abs{\HE_{\rho}u^{h}}_{H^{1}}^{2}}{\norm{u^{h}}_{\Gamma}^{2}}\lesssim \bar{\theta}_{j}^{h}+\rho_{0}H\lesssim \theta_{j}+\,\delta_{j,h}^{2}+\rho_{0}H
        \end{equation}
        due to \cref{lapext,esttheta}. By combining \cref{est33,est34},
        \begin{equation}
            \label{eiggapbbsk2}
            \lim_{h\to0}\abs{\bar{\theta}_{j}^{h}\!-\theta_{j}^{h}}\lesssim \lim_{h\to0}\bigg(\max_{u^{h}\in U_{j}^{h}}\Big\{\frac{\abs{\HE_{\rho}u^{h}}_{H^{1}}^{2}}{\norm{u^{h}}_{\Gamma}^{2}}\inf_{v^{h}\in V_{B,\rho}^{h}}\frac{\abs{\HE_{\rho}u^{h}\!-v^{h}}_{H^{1}}^{2}}{\abs{\HE_{\rho}u^{h}}_{H^{1}}^{2}}\Big\}\bigg)=0.
        \end{equation}
        Combining \cref{eiggapbbsk1,eiggapbbsk2}, we finish the proof.
    \end{proof}
    \begin{corollary}
        \label{estcs}
        Let $(\theta_{\lambda}^{h})^{(j)}$ and $\theta_{\lambda}^{(j)}$ be the $j$th smallest eigenvalue of $S_{\lambda}^{h}$ and $S_{\lambda}$ respectively, then
        \begin{equation*}
            \lim_{h\to0}|c_{s}^{h}-c_{s}|=0,
        \end{equation*}
        where $c_{s}^{h} = \frac{(\theta_{\lambda}^{h})^{(2)}}{(\theta_{\lambda}^{h})^{(3)}}$ and $c_{s} = \frac{\theta_{\lambda}^{(2)}}{\theta_{\lambda}^{(3)}}$.
    \end{corollary}

    Now we can give the convergence factor for Algorithm~\ref{algo}.
    \begin{theorem}
        \label{mainthm}
        Suppose that the coarse mesh size $H$ is small enough, there exists a constant $C$ independent of $h$ and $H$ such that
        \begin{equation*}
            \epsilon_{N}\leq C\epsilon^{2},
        \end{equation*}
        where $\epsilon$ and $\epsilon_{N}$ are errors of the eigenvalue before and after one iteration.
    \end{theorem}
    \begin{proof}
        This proof is mainly based on \cref{cor2}. First, some conditions in \cref{lemgapmulti} need to be verified. According to \cref{minevalS}, the constant $c_{t}$ in \cref{lemgapmulti} satisfies $c_{t}\gtrsim H$. Due to \cref{lapeigest,lemgapmulti,estcs}, the constant $c_{g}$ satisfies
        \begin{equation}
            \label{estcg}
            \begin{aligned}
                c_{g}&=\min_{\rho\in(\lambda^{h},\,\rho_{0})}\Big(\theta_{\rho}^{(3)}-\theta_{\rho}^{(2)}\Big)\geq \frac{\widehat{\lambda}^{h}\!-\rho_{0}}{\widehat{\lambda}^{h}}\,c_{s}^{h}c_{t}
                \\&=\frac{\widehat{\lambda}-\lambda+(\widehat{\lambda}^{h}\!-\widehat{\lambda})+(\lambda-\rho_{0})}{\widehat{\lambda}+(\widehat{\lambda}^{h}\!-\widehat{\lambda})}\,c_{s}^{h}c_{t}
                \\&\gtrsim\frac{\widehat{\lambda}-\lambda+h^{2}+H^{2}}{\widehat{\lambda}+h^{2}}\,c_{s}H\gtrsim H.
            \end{aligned}
        \end{equation}
        By using \cref{lapext,main1,lapeigest,lapcoer}, the convergence factor is bounded by
        \begin{equation}
            \label{estthm}
            \frac{\epsilon_{N}}{\epsilon^{2}}\leq \frac{1}{\alpha-\rho_{0}}+\frac{c_{\HE}^{2}}{c_{g}}\lesssim H^{2}+1,
        \end{equation}
        which finishes the proof.
    \end{proof}

    \subsection{A sharper bound for the convergence factor}
     After  carefully checking  the numerical results in next section, we find the estimation in \cref{mainthm} can be further improved.
     In this part, we will give a sharper estimation for the convergence factor by using a special norm $\norm{\cdot}_{\Gamma}$ on $\Gamma$, which is spectral equivalent to $\norm{\cdot}_{L^{2}(\Gamma)}$. Let $V_{\Gamma,\rho}^{h}$ be the $S_{\rho}^{h}$\nobreakdash-orthogonal space of $u_{\rho}^{h}$ in $V_{\Gamma}^{h}$, for all $u^{h}\in V_{\Gamma,\rho}^{h}$, we have
    \begin{equation*}
        \dual{S_{\rho}^{h}u^{h},u^{h}}-\theta_{\rho}^{h}(u^{h},u^{h})_{\Gamma}\gtrsim H\, (u^{h},u^{h})_{\Gamma}
    \end{equation*}
    due to the estimation for $c_{g}$ in \cref{estcg}. In $V_{\Gamma,\rho}^{h}$, the operator $(S_{\rho}^{h}-\theta_{\rho}^{h})^{-1}$ is symmetric positive definite, and it is spectral equivalent to $(S_{0}^{h})^{-1}$, which means
    \begin{equation}
        \label{speq}
        \bigdual{(S_{\rho}^{h}-\theta_{\rho}^{h}I^{h})^{-1}u^{h},u^{h}}\approx
        \bigdual{(S_{0}^{h})^{-1}u^{h},u^{h}}
    \end{equation}
    holds for all $u^{h}\in V_{\Gamma,\rho}^{h}$, where $I^{h}$ is the identity operator in $V_{\Gamma}^{h}$. Now  we need some important results in non-overlapping domain decomposition methods, whose detailed proof can be found in many papers and books, for examples, \cite{Xu1998,Toselli2005,Mathew2008}.
    \begin{proposition}
        \label{ddm}
        For $\Omega\in\mathbb{R}^{d}$, where $d=2$ or $3$, there exists a decomposition
        \begin{equation*}
            V_{\Gamma}^{h}=R_{H}^{\Ttran}V_{\Gamma}^{H}+\sum_{i=1}^{n_{l}}R_{i}^{\Ttran}V_{\Gamma}^{i},
        \end{equation*}
        where $R_{H}^{\Ttran}$ and $R_{i}^{\Ttran}$'s are interpolation operators, such that following properties hold.
        \begin{itemize}
            \item The space $R_{H}^{\Ttran}V_{\Gamma}^{H}$ is a global coarse space, and all functions in $R_{H}^{\Ttran}V_{\Gamma}^{H}$ are linear in coarse meshes. Other spaces, \ie $R_{i}^{\Ttran}V_{\Gamma}^{i}$'s, are local spaces.
            \item Let $M^{h}$ be the preconditioner defined as
            \begin{equation*}
                M^{h}\defi R_{H}^{\Ttran}(R_{H}S_{0}^{h}R_{H}^{\Ttran})^{-1}R_{H}+\sum_{i=1}^{n_{l}}R_{i}^{\Ttran}(R_{i}S_{0}^{h}R_{i}^{\Ttran})^{-1}R_{i},
            \end{equation*}
            then for all $u^{h}\in V_{\Gamma}^{h}$,
            \begin{equation*}
                \dual{S_{0}^{h}u^{h},u^{h}}\lesssim \bigl(1+\ln(H/h)\bigr)^{2}\dual{S_{0}^{h}u^{h},M^{h}S_{0}^{h}u^{h}}.
            \end{equation*}
            \item For all $u^{h}\in V_{\Gamma}$,
            \begin{equation}
                \label{loc}
                \dual{M^{h}u^{h},u^{h}}\lesssim H^{-1}\norm{R_{H}u^{h}}_{*}^{2}+H\norm{u^{h}}_{\Gamma}^{2},
            \end{equation}
            where $\norm{R_{H}u^{h}}_{*}$ is defined as
            \begin{equation*}
                \norm{R_{H}u^{h}}_{*}\defi \sup_{0\neq u_{H}^{h}\in R_{H}^{\Ttran}V_{\Gamma}^{H}}\frac{(u^{h},u_{H}^{h})_{\Gamma}}{\norm{u_{H}^{h}}_{\Gamma}}.
            \end{equation*}
        \end{itemize}
    \end{proposition}
    \begin{remark}
        The estimation \cref{loc} can be obtained by the Poincar\'e inequality and scaling arguments when $R_{i}^{\Ttran}V_{\Gamma}^{i}$'s are local spaces with diameters no more than $H$.
    \end{remark}

    Let $Q_{H}$ be the $L^{2}(\Gamma)$ projection from $V_{\Gamma}^{h}$ to $R_{H}^{\Ttran}V_{\Gamma}^{H}$, \ie
    \begin{equation}
        \label{defQ}
        (Q_{H}u^{h},u_{H}^{h})_{L^{2}(\Gamma)}=(u^{h},u_{H}^{h})_{L^{2}(\Gamma)}\quad \forall\, u_{H}^{h}\in R_{H}^{\Ttran}V_{\Gamma}^{H}.
    \end{equation}
    From the definition we know
    \begin{equation}
        \label{defRhQ}
        \norm{R_{H}r_{\rho}^{h}}_{*}=\sup_{0\neq u_{H}^{h}\in R_{H}^{\Ttran}V_{\Gamma}^{H}}\frac{(r_{\rho}^{h},u_{H}^{h})_{\Gamma}}{\norm{u_{H}^{h}}_{\Gamma}}=\norm{Q_{H}r_{\rho}^{h}}_{\Gamma}.
    \end{equation}
    Now, we can define a bilinear form $(\cdot,\cdot)_{\Gamma}$ as follows:
    \begin{equation}
        \label{defgamma}
        (u_{1}^{h},u_{2}^{h})_{\Gamma}\defi H^{-1}(\HE_{0}^{h}Q_{H}u_{1}^{h},\HE_{0}^{h}Q_{H}u_{2}^{h})+(u_{1}^{h}-Q_{H}u_{1}^{h},u_{2}^{h}-Q_{H}u_{2}^{h})_{L^{2}(\Gamma)}.
    \end{equation}
    It can be verified that $(\cdot,\cdot)_{\Gamma}$ is an inner product on $V_{\Gamma}^{h}$. Moreover, $Q_{H}$ is also a $(\cdot,\cdot)_{\Gamma}$\nobreakdash - projection, and
    \begin{equation}
        \label{eqnorm}
        H^{1/2}\norm{Q_{H}u^{h}}_{\Gamma}=\norm{\HE_{0}^{h}Q_{H}u^{h}}
    \end{equation}
    holds for all $u^{h}\in V_{\Gamma}^{h}$, where $\norm{\cdot}_{\Gamma}$ is the norm corresponding to $(\cdot,\cdot)_{\Gamma}$.
    \begin{lemma}
        For all $u^{h}\in V_{\Gamma}^{h}$, $\norm{u^{h}}_{\Gamma}\approx\norm{u^{h}}_{L^{2}(\Gamma)}$.
    \end{lemma}
    \begin{proof}
        By \cref{defgamma}, it is sufficient to prove that for all $u_{H}^{h}\in R_{H}^{\Ttran}V_{\Gamma}^{h}$,
        \begin{equation*}
            \norm{u_{H}^{h}}_{L^{2}(\Gamma)}\approx \norm{u_{H}^{h}}_{\Gamma}=H^{-1/2}\norm{\HE_{0}^{h}u_{H}^{h}}.
        \end{equation*}
        On the one hand, by \cref{lapext}, we have
        \begin{equation*}
            \norm{\HE_{0}^{h}u_{H}^{h}} \lesssim H^{1/2}\norm{u_{H}^{h}}_{L^{2}(\Gamma)}.
        \end{equation*}
        On the other hand, according to Theorem~1.6.6 of \cite{Brenner2008} and the Cauchy-Schwarz inequality,
        \begin{equation*}
            \begin{aligned}
                \norm{u_{H}^{h}}_{L^{2}(\Gamma)}^{2}&\lesssim \sum_{k=1}^{N}\norm{\HE_{0}^{h}u_{H}^{h}}_{L^{2}(\Omega_{k})}\norm{\HE_{0}^{h}u_{H}^{h}}_{H^{1}(\Omega_{k})}\\
                &\leq \norm{\HE_{0}^{h}u_{H}^{h}}(\abs{\HE_{0}^{h}u_{H}^{h}}_{H^{1}}^{2}+H^{-2}\norm{\HE_{0}^{h}u_{H}^{h}}^{2})^{1/2}.
            \end{aligned}
        \end{equation*}
        By the inverse estimation, \cref{minrq,eqh1norm},
        \begin{equation*}
            \abs{\HE_{0}^{h}u_{H}^{h}}_{H^{1}}^{2} \lesssim \sum_{k=1}^{N}\abs{u_{H}^{h}}_{H^{1/2}(\partial\Omega_{k})}^{2} \lesssim H^{-1}\norm{u_{H}^{h}}_{L^{2}(\Gamma)}^{2}.
        \end{equation*}
        Combining these two inequalities above, we have
        \begin{equation*}
            \norm{u_{H}^{h}}_{L^{2}(\Gamma)}\lesssim H^{-1/2}\norm{\HE_{0}^{h}u_{H}^{h}},
        \end{equation*}
        which finishes the proof.
    \end{proof}

    \begin{lemma}
        \label{PiH}
        Let $\Pi_{H}$ denote the interpolation operator associated with the coarse space $R_{H}^{\Ttran}V_{\Gamma}^{H}$, for all $\lambda^{h}\leq\rho\leq\rho_{0}$,
        \begin{equation*}
            \bigabs{\norm{\Pi_{H}u_{\rho}^{h}}_{\Gamma}-\norm{u_{\rho}^{h}}_{\Gamma}}\lesssim H\norm{u_{\rho}^{h}}_{\Gamma}.
        \end{equation*}
    \end{lemma}
    \begin{proof}
        By the estimation of interpolation and the trace theorem, we have
        \begin{equation*}
            \bigabs{\norm{\Pi_{H}u_{\rho}^{h}}_{\Gamma}-\norm{u_{\rho}^{h}}_{\Gamma}}^{2}\leq
            \norm{\Pi_{H}u_{\rho}^{h}-u_{\rho}^{h}}_{\Gamma}^{2}\lesssim
            H\sum_{k=1}^{N}\abs{u_{\rho}^{h}}_{H^{1/2}(\partial\Omega_{k})}^{2}
            \lesssim H\,\abs{\HE_{\rho}^{h}u_{\rho}^{h}}_{H^{1}(\Omega)}^{2}.
        \end{equation*}
        According to \cref{eqh1norm,lapext},
        \begin{equation*}
            \abs{\HE_{\rho}^{h}u_{\rho}^{h}}_{H^{1}}^{2}\lesssim \abs{\HE_{\rho}^{h}u_{\rho}^{h}}_{a}^{2}=\theta_{\rho}^{h}\norm{u_{\rho}^{h}}_{\Gamma}^{2}+\rho\norm{\HE_{\rho}^{h}u_{\rho}^{h}}^{2}\lesssim H\,\norm{u_{\rho}^{h}}_{\Gamma}^{2}.
        \end{equation*}
        The lemma is proved by combining these two inequalities above.
    \end{proof}
    \begin{corollary}
        \label{Qurho}
        Since $Q_{H}$ is a projection, $\bigabs{\norm{Q_{H}u_{\rho}^{h}}_{\Gamma}-\norm{u_{\rho}^{h}}_{\Gamma}}\lesssim H\norm{u_{\rho}^{h}}_{\Gamma}$.
    \end{corollary}
    \begin{lemma}
        \label{sharp}
        By using notations in \cref{main1}, in the finite element space,
        \begin{equation*}
            \bigdual{(S_{\rho}^{h}-\theta_{\rho}^{h}I^{h})(u_{\rho}^{h})^{\prime},(u_{\rho}^{h})^{\prime}} \lesssim H^{2}\bigl(1+\ln(H/h)\bigr)^{2} \bignorm{v_{\rho}^{h}}^{2}
        \end{equation*}
        holds for all $\lambda^{h}\leq\rho\leq \rho_{0}$, where $v_{\rho}^{h}=\HE_{\rho}^{h}u_{\rho}^{h}$ and $\norm{u_{\rho}^{h}}_{\Gamma}=1$.
    \end{lemma}
    \begin{proof}
        Let $r_{\rho}^{h}=(\theta_{\rho}^{h})^{\prime}u_{\rho}^{h}-(S_{\rho}^{h})^{\prime}u_{\rho}^{h}$, according to \cref{speq,ddm,defRhQ}, we have
        \begin{equation*}
            \begin{aligned}
            \bigdual{(S_{\rho}^{h}-\theta_{\rho}^{h}I^{h})(u_{\rho}^{h})^{\prime},(u_{\rho}^{h})^{\prime}}
            &=\bigdual{(S_{\rho}^{h}-\theta_{\rho}^{h}I^{h})^{-1}r_{\rho}^{h},r_{\rho}^{h}}
            \approx\bigdual{(S_{0}^{h})^{-1}r_{\rho}^{h},r_{\rho}^{h}}\\
            &=\bigdual{S_{0}^{h}(S_{0}^{h})^{-1}r_{\rho}^{h},(S_{0}^{h})^{-1}r_{\rho}^{h}}
            \lesssim \bigl(1+\ln(H/h)\bigr)^{2}\bigdual{M^{h}r_{\rho}^{h},r_{\rho}^{h}} \\
            &\lesssim  \bigl(1+\ln(H/h)\bigr)^{2} \Bigl(H^{-1}\norm{Q_{H}r_{\rho}^{h}}_{\Gamma}^{2}+H\,\norm{r_{\rho}^{h}}_{\Gamma}^{2}\Bigr).
            \end{aligned}
        \end{equation*}
        By \cref{lemdiff1,defSprime,eqevec},
        \begin{equation*}
            \norm{r_{\rho}^{h}}_{\Gamma}\leq \abs{(\theta_{\rho}^{h})^{\prime}}\norm{u_{\rho}^{h}}_{\Gamma}+\norm{v_{\rho}^{h}}\normm{\HE_{\rho}^{h}}\lesssim H.
        \end{equation*}
        Combining these two inequalities above, we have
        \begin{equation}
            \label{sharp1}
            \bigdual{(S_{\rho}^{h}-\theta_{\rho}^{h}I^{h})(u_{\rho}^{h})^{\prime},(u_{\rho}^{h})^{\prime}} \lesssim
            \bigl(1+\ln(H/h)\bigr)^{2} \Bigl(H^{-1}\norm{Q_{H}r_{\rho}^{h}}_{\Gamma}^{2}+H^{3}\Bigr).
        \end{equation}
        According to \cref{defSprime,lapext,Qurho},
        \begin{equation}
            \label{sharp3}
            \begin{aligned}
                \norm{Q_{H}r_{\rho}^{h}}_{\Gamma}
            &\leq \bignorm{(\theta_{\rho}^{h})^{\prime}Q_{H}u_{\rho}^{h}-Q_{H}(S_{\rho}^{h})^{\prime}Q_{H}u_{\rho}^{h}}_{\Gamma}
            +\bignorm{Q_{H}(S_{\rho}^{h})^{\prime}(u_{\rho}^{h}-Q_{H}u_{\rho}^{h})}_{\Gamma}\\
            &\lesssim \bignorm{(\theta_{\rho}^{h})^{\prime}Q_{H}u_{\rho}^{h}-Q_{H}(S_{\rho}^{h})^{\prime}Q_{H}u_{\rho}^{h}}_{\Gamma}
            + H^{2}.
            \end{aligned}
        \end{equation}
        For the first term, due to \cref{continuity,lapext},
        \begin{equation}
            \label{sharp4}
            \begin{aligned}
                &\bignorm{(\theta_{\rho}^{h})^{\prime}Q_{H}u_{\rho}^{h}-Q_{H}(S_{\rho}^{h})^{\prime}Q_{H}u_{\rho}^{h}}_{\Gamma}\\
                \leq& \bignorm{(\theta_{\rho}^{h})^{\prime}Q_{H}u_{\rho}^{h}-Q_{H}(S_{0}^{h})^{\prime}Q_{H}u_{\rho}^{h}}_{\Gamma}
                + \bignorm{Q_{H}\bigl((S_{0}^{h})^{\prime}-(S_{\rho}^{h})^{\prime}\bigr)Q_{H}u_{\rho}^{h}}_{\Gamma}\\
                \lesssim & \bignorm{(\theta_{\rho}^{h})^{\prime}Q_{H}u_{\rho}^{h}-Q_{H}(S_{0}^{h})^{\prime}Q_{H}u_{\rho}^{h}}_{\Gamma}+H^{3}.
            \end{aligned}
        \end{equation}
        By \cref{defSprime,defgamma}
        \begin{equation}
            \label{sharp5}
            \begin{aligned}
                &\bignorm{(\theta_{\rho}^{h})^{\prime}Q_{H}u_{\rho}^{h}-Q_{H}(S_{0}^{h})^{\prime}Q_{H}u_{\rho}^{h}}_{\Gamma}^{2}\\
                =& \bigl((\theta_{\rho}^{h})^{\prime}\bigr)^{2}\norm{Q_{H}u_{\rho}^{h}}_{\Gamma}^{2}
                +2(\theta_{\rho}^{h})^{\prime}\norm{\HE_{0}^{h}Q_{H}u_{\rho}^{h}}^{2}
                +\sup_{0\neq u^{h}\in V_{\Gamma}^{h}}\frac{\bigabs{(\HE_{0}^{h}Q_{H}u_{\rho}^{h},\HE_{0}^{h}Q_{H}u^{h})}^{2}}{\norm{u^{h}}_{\Gamma}^{2}}\\
                \leq &\Bigl(\bigl((\theta_{\rho}^{h})^{\prime}\bigr)^{2}+2(\theta_{\rho}^{h})^{\prime}H+H^{2} \Bigr)\norm{Q_{H}u_{\rho}^{h}}_{\Gamma}^{2}\leq \bigl((\theta_{\rho}^{h})^{\prime}+H\bigr)^{2}.
            \end{aligned}
        \end{equation}
        Using \cref{lemdiff1,eqnorm,continuity,lapext}, we know that
        \begin{equation}
            \label{sharp6}
            \begin{aligned}
                \bigabs{(\theta_{\rho}^{h})^{\prime}+H}
            \leq& \bigabs{\norm{\HE_{\rho}^{h}u_{\rho}^{h}}^{2}-\norm{\HE_{\rho}^{h}Q_{H}u_{\rho}^{h}}^{2}}+
            \bigabs{\norm{\HE_{\rho}^{h}Q_{H}u_{\rho}^{h}}^{2}-\norm{\HE_{0}^{h}Q_{H}u_{\rho}^{h}}^{2}}\\
            &+H\bigabs{\norm{Q_{H}u_{\rho}^{h}}_{\Gamma}^{2}-\norm{u_{\rho}^{h}}_{\Gamma}^{2}}\lesssim H^{2}.
            \end{aligned}
        \end{equation}
        Thus the lemma is proved by \cref{eqevec,sharp1,sharp1,sharp3,sharp4,sharp5,sharp6}.
    \end{proof}

    Combining \cref{sharp} with \cref{main1}, we obtain a sharper estimation for the convergence factor.
    \begin{theorem}
        \label{mainsharp}
        Suppose the coarse mesh size $H$ is small enough, if the inner product $(\cdot,\cdot)_{\Gamma}$ is defined as \cref{defgamma}, there exists a constant $C$ independent of $h$ and $H$ such that
        \begin{equation*}
            \epsilon_{N}\leq CH^{2}\bigl(1+\ln(H/h)\bigr)^{2}\epsilon^{2},
        \end{equation*}
        where $\epsilon$ and $\epsilon_{N}$ are errors of the eigenvalue before and after one iteration.
    \end{theorem}
    \begin{remark}
        In this part, we only prove that when $(\cdot,\cdot)_{\Gamma}$ is defined as \cref{defgamma}, the rate of convergence is $\epsilon_{N}\leq CH^{2}(1+\ln(H/h))^{2}\epsilon^{2}$. For other inner products, whose corresponding norms are spectral equivalent to $\norm{\cdot}_{L^{2}(\Gamma)}$, whether similar results can be obtained is still unknown.  We do not know similar results before, and the discussions along this direction is quite interesting.
    \end{remark}

    \section{Numerical experiments}
    In this section, we present some numerical results to support our theoretical analysis above. We  compute some second order symmetric elliptic eigenvalue problems in 2D and 3D by Algorithm~\ref{algo}. Assume $\MA$ and $\MM$ are the stiffness matrix and mass matrix generated by the finite element method as \cref{femsec} respectively. Due to the non-overlapping domain decomposition method, $\MA$ and $\MM$ can be partitioned as
    \begin{equation*}
        \MA=\begin{bmatrix}
            \MA_{II}&\MA_{IB}\\
            \MA_{BI}&\MA_{BB}
        \end{bmatrix}\quad\text{and}\quad
        \MM=\begin{bmatrix}
            \MM_{II}&\MM_{IB}\\
            \MM_{BI}&\MM_{BB}
        \end{bmatrix},
    \end{equation*}
    where $\MA_{IB}=(\MA_{BI})^{\Ttran}$ and $\MM_{IB}=(\MM_{BI})^{\Ttran}$,  indices $I$ are associated with the nodes in  $\cup_{i=1}^{N}\Omega_{i}$ while $B$ are associated with the nodes on $\Gamma$. Let
    \begin{equation}
        \label{mSchur}
        \MS_{\rho}\defi(\MA_{BB}-\rho\, \MM_{BB})-(\MA_{BI}-\rho\,\MM_{BI})\,(\MA_{II}-\rho\,\MM_{II})^{-1}(\MA_{IB}-\rho\,\MM_{IB}),
    \end{equation}
    the eigenvalue problem $\MA \Mv=\lambda^{h}\,\MM\Mv$ can be rewritten as $\MS_{\lambda^{h}}\Mu=0$, where $\Mu$ is the restriction of $\Mv$ on $\Gamma$. Suppose $\MM_{\Gamma}=h^{d-1}\,\MI_{\Gamma}$ is the mass matrix on $\Gamma$, where $\MI$ is identity matrix on $\Gamma$, the nonlinear eigenvalue problem can be written as
    \begin{equation}
        \label{mschureig}
        \MS_{\rho}\Mu_{\rho}=\theta_{\rho}^{h}\,\MM_{\Gamma}\Mu_{\rho}.
    \end{equation}
    By defining the extension operator from $\Gamma$ to $\Omega$ as
    \begin{equation*}
        \MH_{\rho}=\begin{bmatrix}
            -(\MA_{II}-\rho\,\MM_{II})^{-1}(\MA_{IB}-\rho\,\MM_{IB})\\
            \MI
        \end{bmatrix},
    \end{equation*}
    we get the matrix version of Algorithm~\ref{algo}. Our numerical experiments were performed in Matlab 2020b, the real solution $\lambda^{h}$ is calculated from solving the smallest eigenvalue of $(\MA,\MM)$ by the MATLAB function ``\textbf{eigs}'' with tolerance $10^{-15}$, the stopping criterion is chosen as $\epsilon<10^{-12}$ and the subproblem \cref{mschureig} is solved by the Matlab function ``\textbf{eigs}'' with tolerance $10^{-12}$. Let $\rho_{k}$ and $\eta_{k}$ be the approximated eigenvalue and convergence factor after $k$ steps of Newton's method respectively, \ie
    \begin{equation*}
        \epsilon_{k}=\rho_{k}-\lambda^{h}\quad\text{and}\quad\eta_{k}=\frac{\epsilon_{k+1}}{\epsilon_{k}^{2}}.
    \end{equation*}
    In our experiments, Algorithm~\ref{algo} converges after few steps, so we only consider $\eta_{0}$.
    \begin{figure}
        \subfloat[3D case]{
            \label{figdomain3}
            \includegraphics[width=\figsizeD]{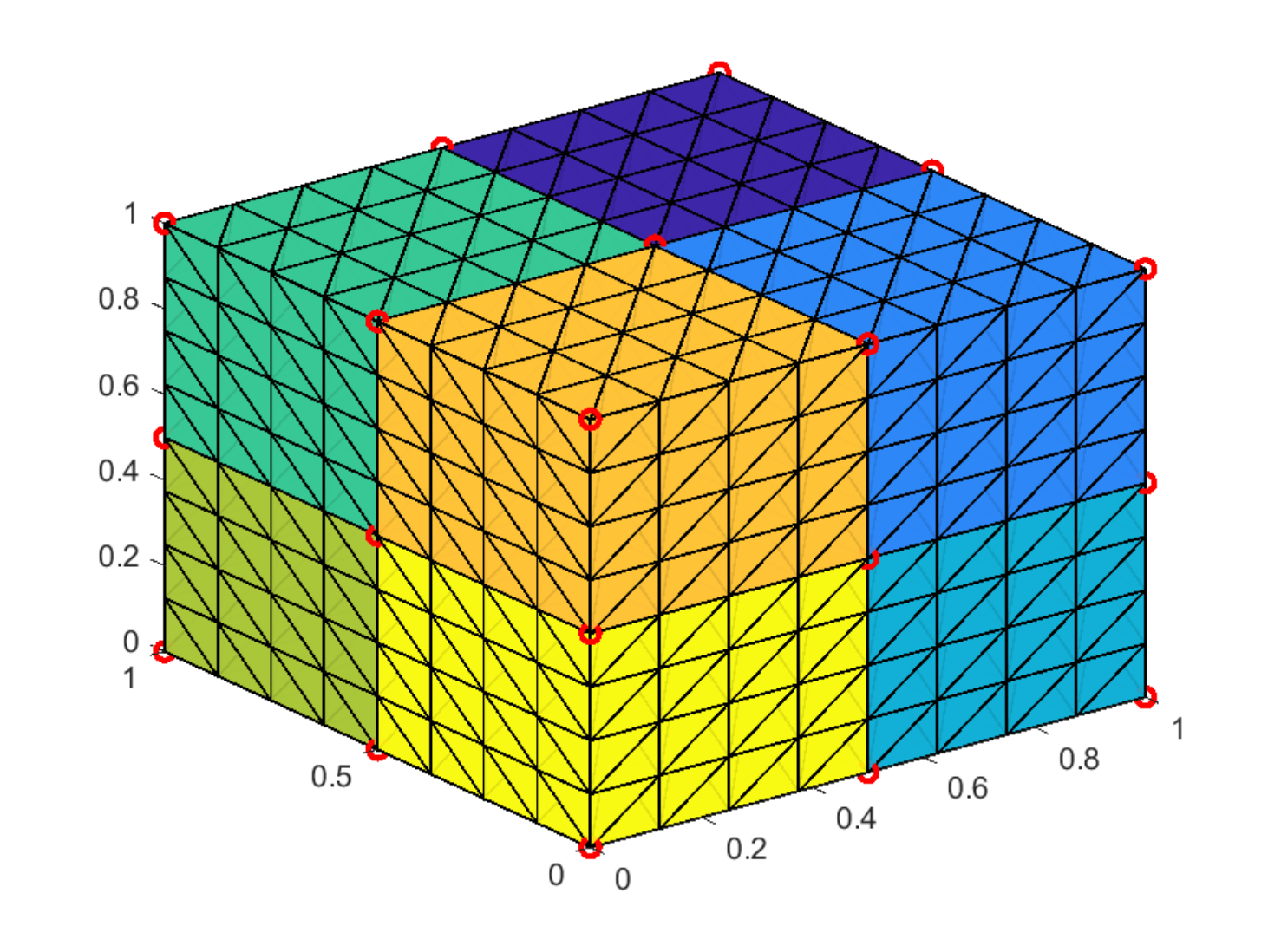}
        }
        \subfloat[2D L-shaped case]{
            \label{figdomain2}
            \includegraphics[width=\figsizeD]{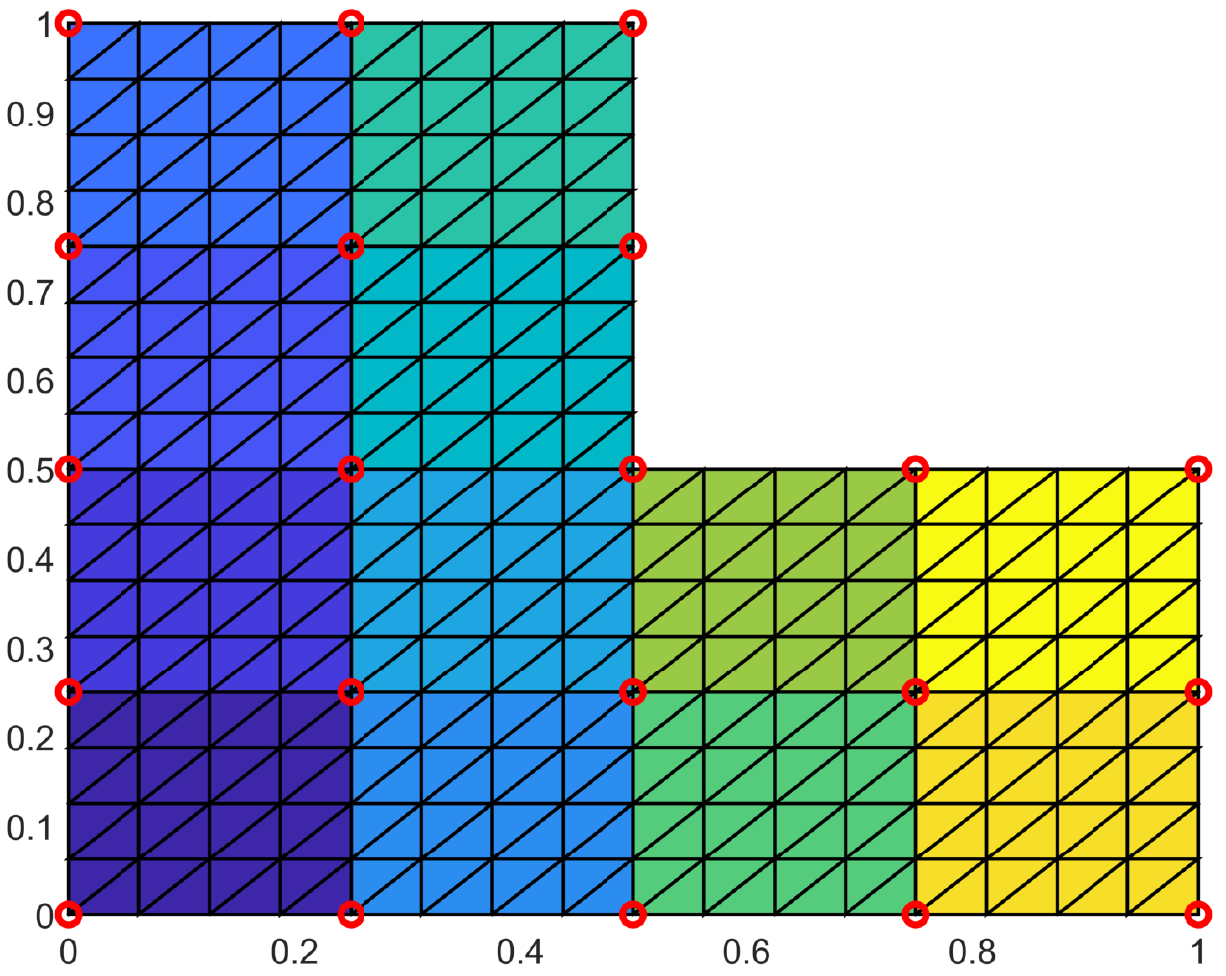}
        }
        \caption{A partition for non-overlapping domain decomposition. Red points are nodes for coarse mesh. Meshes in $\mathcal{T}^{h}$ with same color belong to a same subdomain.}
    \end{figure}
    \subsection{The Laplacian eigenvalue problem in 3D}
    In this subsection, the domain $\Omega$ is the unit cube $[0,1]^{3}$ in 3D, and a partition is shown in \cref{figdomain3}. We consider the relationship between the convergence factor $\eta$ and the fine mesh size $h$ or the coarse mesh size $H$ separately. For the relationship with the fine mesh size, the coarse mesh size is fixed as $H=2^{-1}$ and fine mesh sizes $h$ are chosen as $2^{-j}$ for $j=2,\dotsc,5$. For the relationship with the coarse mesh size, the fine mesh size is fixed as $h=2^{-5}$ and the coarse mesh sizes $H$ are chosen as $2^{-j}$ for $j=1,\dotsc,4$. \Cref{tabhist3D,fighist3D} show that the Newton-Schur method converges quadratically. We can see that the convergence factor $\eta_{0}$ decreases in $\mathcal{O}(H^{2})$ from \cref{figfac3D}, which means the logarithmic factor may be removed.

    \begin{table}[tbhp]
        \centering
        \subfloat[Various fine mesh sizes]{
        \scalebox{0.9}{
            \begin{tabular}{|c|c|c|c|c|c|c|}
                \hline
                $h$ & $\epsilon_{0}$ & $\epsilon_{1}$ & $\epsilon_{2}$ & $\epsilon_{3}$ & $\epsilon_{4}$ & $\epsilon_{5}$ \\ \hline
                $2^{-2}$ & 0.6000 & 0.0220 & 2.5235e-05 & 3.2885e-11 & 3.7896e-16 & $\surd$ \\ \hline
                $2^{-3}$ & 0.9031 & 0.1442 & 0.0030 & 1.2326e-06 & 2.1061e-13 & $\surd$ \\ \hline
                $2^{-4}$ & 0.9943 & 0.2411 & 0.0119 & 2.7295e-05 & 1.4354e-10 & 4.3693e-15\\ \hline
                $2^{-5}$ & 1.0183 & 0.2799 & 0.0183 & 7.4511e-05 & 1.2260e-09 & 2.2229e-14\\ \hline
            \end{tabular}
        }}

        \subfloat[Various coarse mesh sizes]{
        \scalebox{0.9}{
            \begin{tabular}{|c|c|c|c|c|c|c|}
                \hline
                $H$ & $\epsilon_{0}$ & $\epsilon_{1}$ & $\epsilon_{2}$ & $\epsilon_{3}$ & $\epsilon_{4}$ & $\epsilon_{5}$ \\ \hline
                $2^{-1}$ & 1.0183 & 0.2799 & 0.0183 & 7.4511e-05 & 1.2260e-09 & 2.2229e-14\\ \hline
                $2^{-2}$ & 0.2614 & 0.0029 & 3.5313e-07 & 2.7726e-14 & $\surd$ & $\surd$   \\ \hline
                $2^{-3}$ & 0.0605 & 2.7729e-05 & 5.8386e-12 & 2.2946e-14 & $\surd$ & $\surd$ \\ \hline
                $2^{-4}$ & 0.0120 & 1.0350e-07 & 2.1750e-14 & $\surd$  & $\surd$ & $\surd$ \\ \hline
            \end{tabular}
        }}
        \caption{Relative errors for approximated eigenvalues by Algorithm~\ref{algo} for the 3D Laplacian eigenvalue problem. The ``$\surd$'' entry means that the algorithm converged.}
        \label{tabhist3D}
    \end{table}

    \begin{figure}[tbhp]
        \centering
        \subfloat[Various fine mesh sizes]{
            \includegraphics[width=\figsizeD]{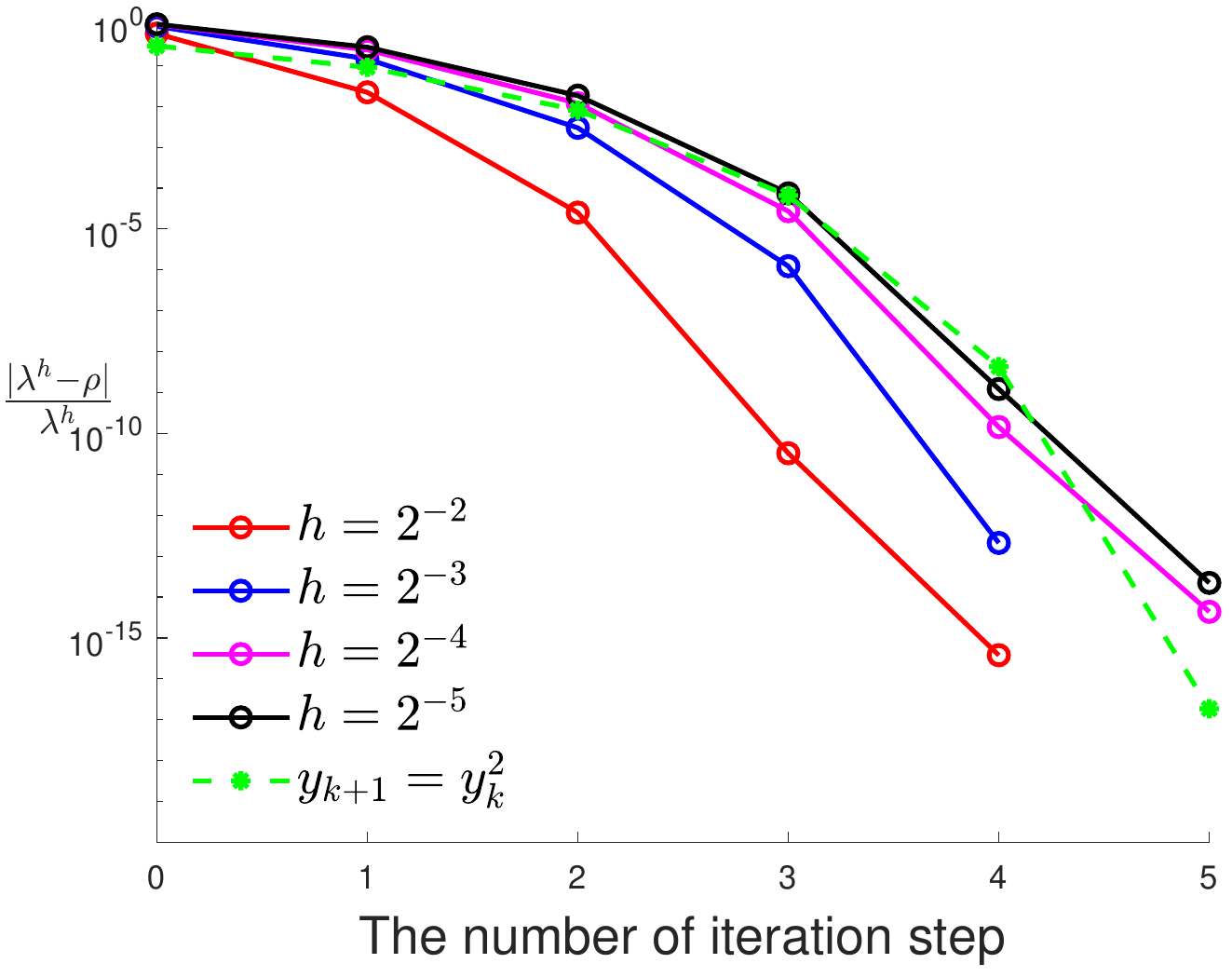}
        }
        \subfloat[Various coarse mesh sizes]{
            \includegraphics[width=\figsizeD]{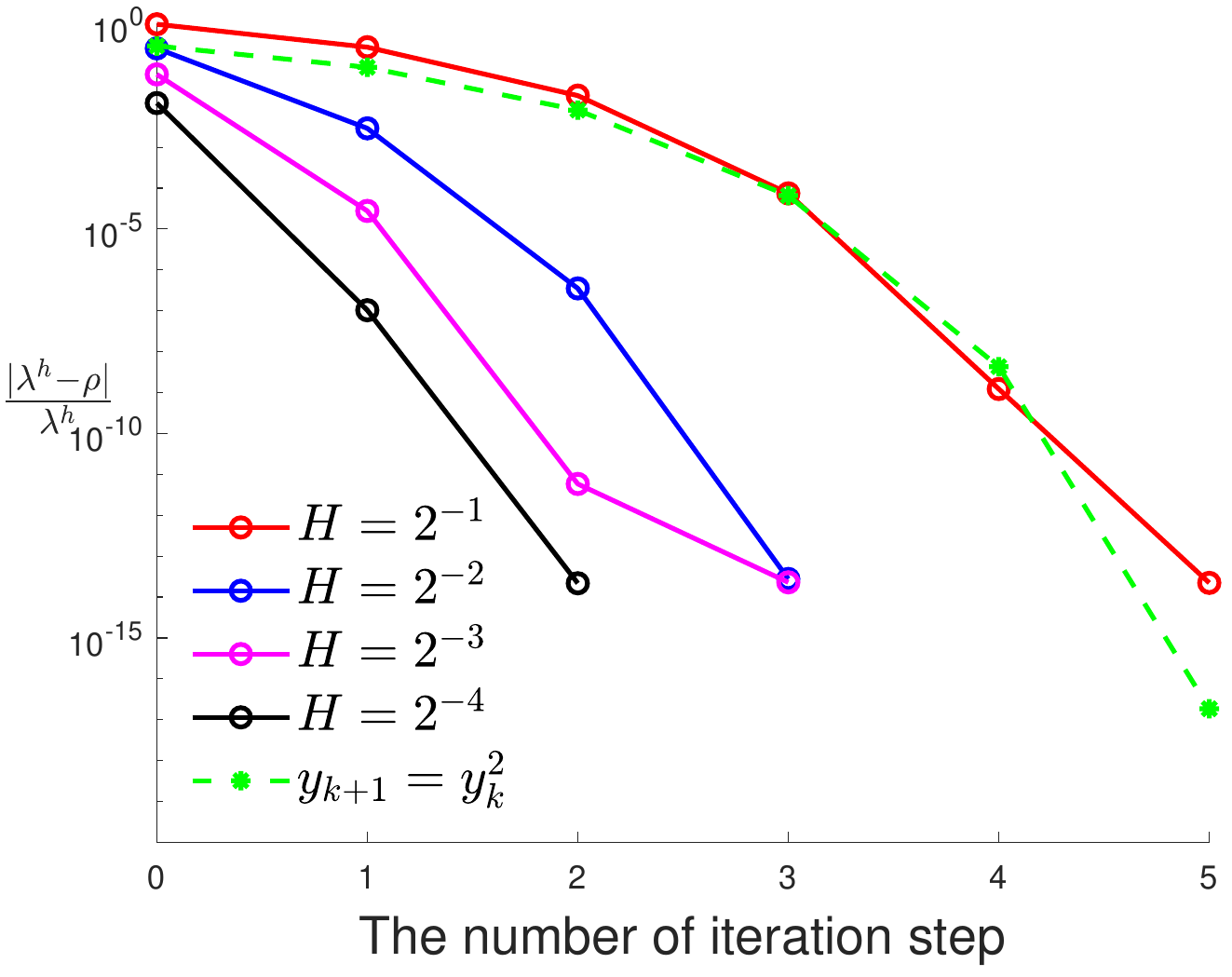}
        }
        \caption{Convergence history of the 3D Laplacian eigenvalue problem. The green dashed lines are references lines with $y_{0}=0.3$.}
        \label{fighist3D}
    \end{figure}
    \begin{figure}[tbhp]
        \centering
        \subfloat[Various fine mesh sizes]{
            \includegraphics[width=\figsizeD]{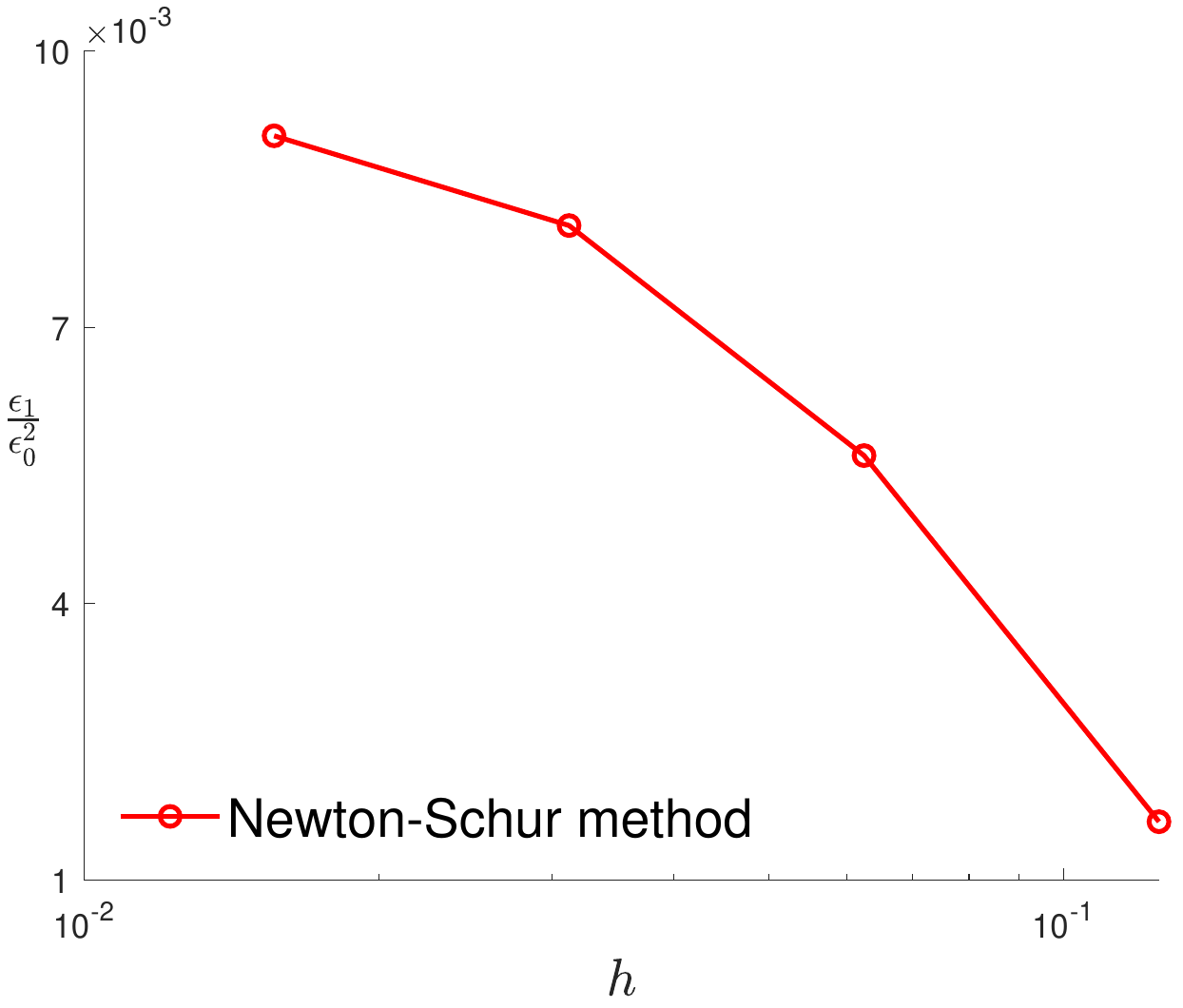}
        }
        \subfloat[Various coarse mesh sizes]{
            \includegraphics[width=\figsizeD]{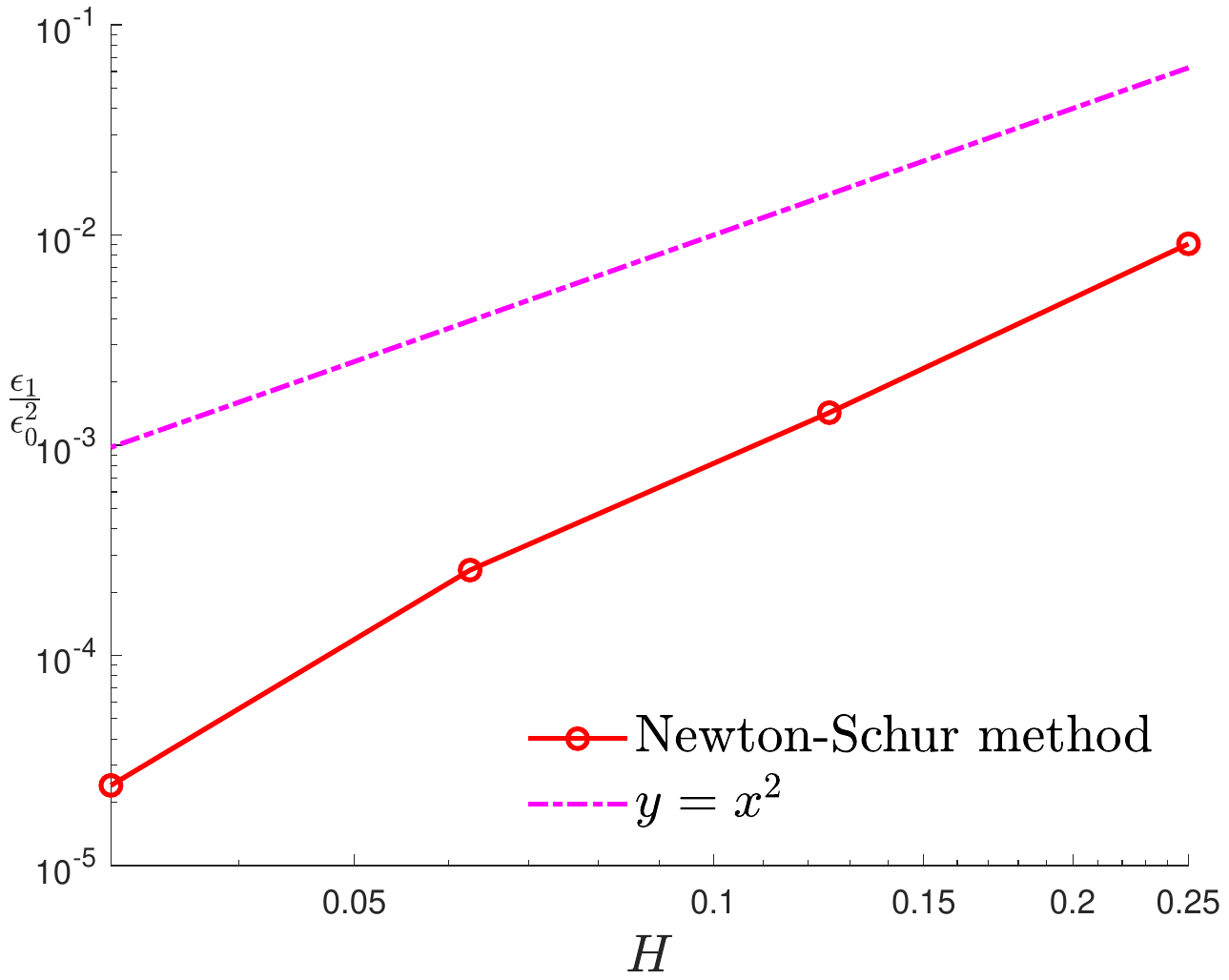}
        }
        \caption{Convergence factor for the 3D Laplacian eigenvalue problem.}
        \label{figfac3D}
    \end{figure}
    \subsection{The Laplacian eigenvalue problem on a 2D L-shaped domain}
    In this subsection, the domain $\Omega$ is an L-shaped domain in 2D, which is shown in \cref{figdomain2}. Similar to the previous subsection, we also consider the relationship between the convergence factor $\eta$ and the fine mesh size $h$ or the coarse mesh size $H$ separately. For the relationship with the fine mesh size, the coarse mesh size is fixed as $H=2^{-2}$ and fine mesh sizes $h$ are chosen as $2^{-j}$ for $j=3,\dotsc,9$. For the relationship with the coarse mesh size, the fine mesh size is fixed as $h=2^{-8}$ and the coarse mesh sizes $H$ are chosen as $2^{-j}$ for $j=2,\dotsc,6$. \Cref{tabhist2D,fighist2D} show that the Newton-Schur method converges quadratically. We can see that the convergence factor $\eta_{0}$ decreases in $\mathcal{O}(H^{2})$ from \cref{figfac2D}, which is similar to the 3D case, even though the domain $\Omega$ is no longer convex.

    \begin{table}[tbhp]
        \centering
        \subfloat[Various fine mesh sizes]{
        \begin{tabular}{|c|c|c|c|c|}
            \hline
            $h$ & $\epsilon_{0}$ & $\epsilon_{1}$ & $\epsilon_{2}$ & $\epsilon_{3}$ \\ \hline
            $2^{-3}$ & 0.0811 & 5.5397e-04 & 2.5657e-08 & 1.7824e-16  \\ \hline
            $2^{-4}$ & 0.1061 & 0.0011 & 1.2506e-07 & 1.4172e-16  \\ \hline
            $2^{-5}$ & 0.1139 & 0.0014 & 2.0185e-07 & 2.0201e-15 \\ \hline
            $2^{-6}$ & 0.1164 & 0.0015 & 2.4069e-07 & 1.4172e-14 \\ \hline
            $2^{-7}$ & 0.1172 & 0.0015 & 2.5877e-07 & 6.8337e-14 \\ \hline
            $2^{-8}$ & 0.1175 & 0.0015 & 2.6718e-07 & 2.8761e-13 \\ \hline
            $2^{-9}$ & 0.1176 & 0.0016 & 2.7114e-07 & 1.1677e-12 \\ \hline
        \end{tabular}}

        \subfloat[Various coarse mesh sizes]{
        \begin{tabular}{|c|c|c|c|c|}
            \hline
            $H$ & $\epsilon_{0}$ & $\epsilon_{1}$ & $\epsilon_{2}$ & $\epsilon_{3}$ \\ \hline
            $2^{-2}$ & 0.1175 & 0.0015 & 2.6718e-07 & 2.8761e-13\\ \hline
            $2^{-3}$ & 0.0337 & 2.8695e-05 & 2.0531e-11 & 2.9258e-13  \\ \hline
            $2^{-4}$ & 0.0103 & 6.5277e-07 & 2.9424e-13 & $\surd$ \\ \hline
            $2^{-5}$ & 0.0033 & 1.5976e-08 & 3.0769e-13 & $\surd$ \\ \hline
            $2^{-6}$ & 0.0010 & 3.6742e-10 & 2.9516e-13 & $\surd$\\ \hline
        \end{tabular}}
        \caption{Relative errors for approximated eigenvalues by Algorithm~\ref{algo} for the 2D Laplacian eigenvalue problem on an L-shaped domain. The ``$\surd$'' entry means that the algorithm converged.}
        \label{tabhist2D}
    \end{table}
    \begin{figure}[tbhp]
        \centering
        \subfloat[Various fine mesh sizes]{
            \includegraphics[width=\figsizeD]{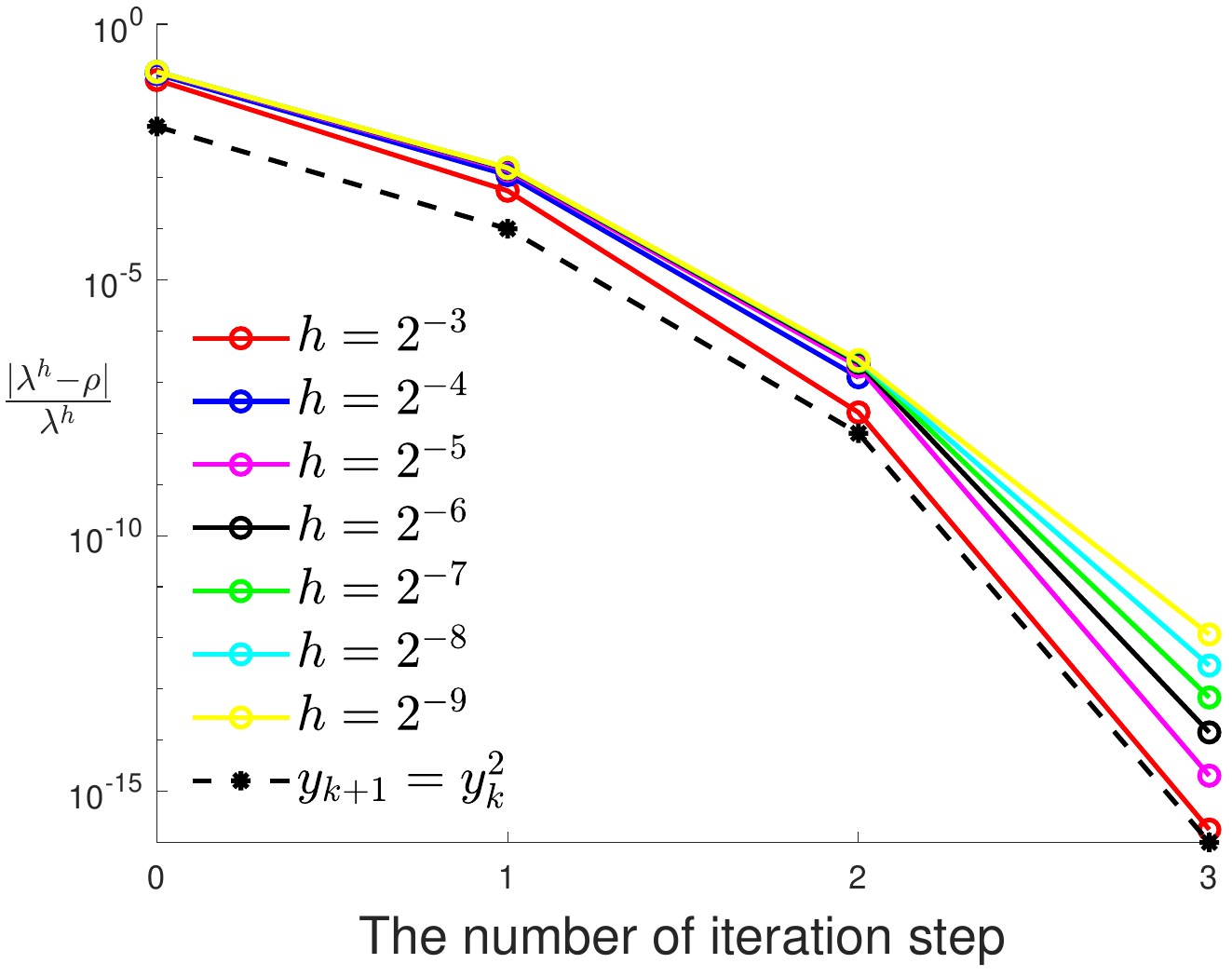}
        }
        \subfloat[Various coarse mesh sizes]{
            \includegraphics[width=\figsizeD]{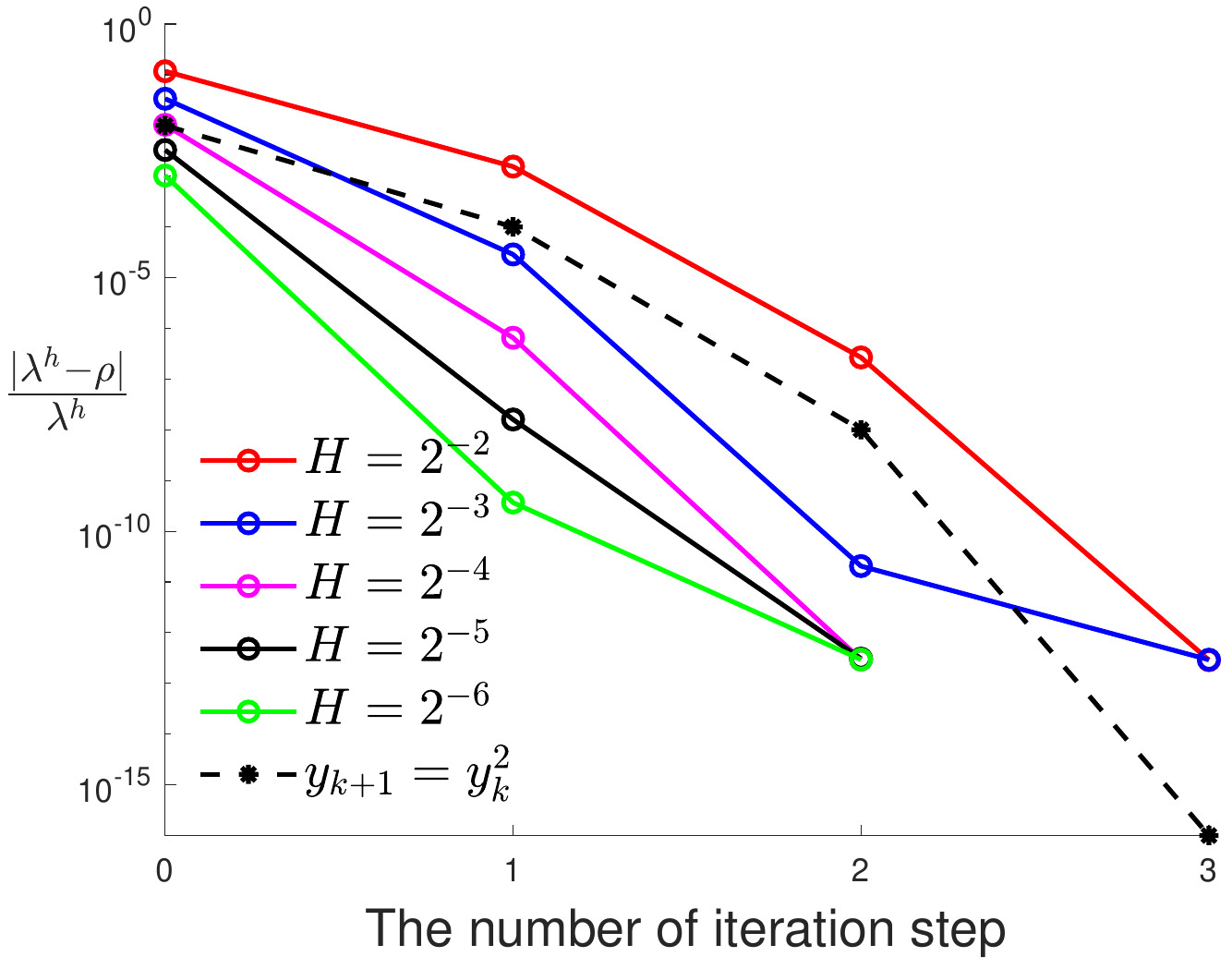}
        }
        \caption{Convergence history of the 2D Laplacian eigenvalue problem on an L-shaped domain. The black dashed lines are references lines with $y_{0}=0.01$.}
        \label{fighist2D}
    \end{figure}
    \begin{figure}[tbhp]
        \centering
        \subfloat[Various fine mesh sizes]{
            \includegraphics[width=\figsizeD]{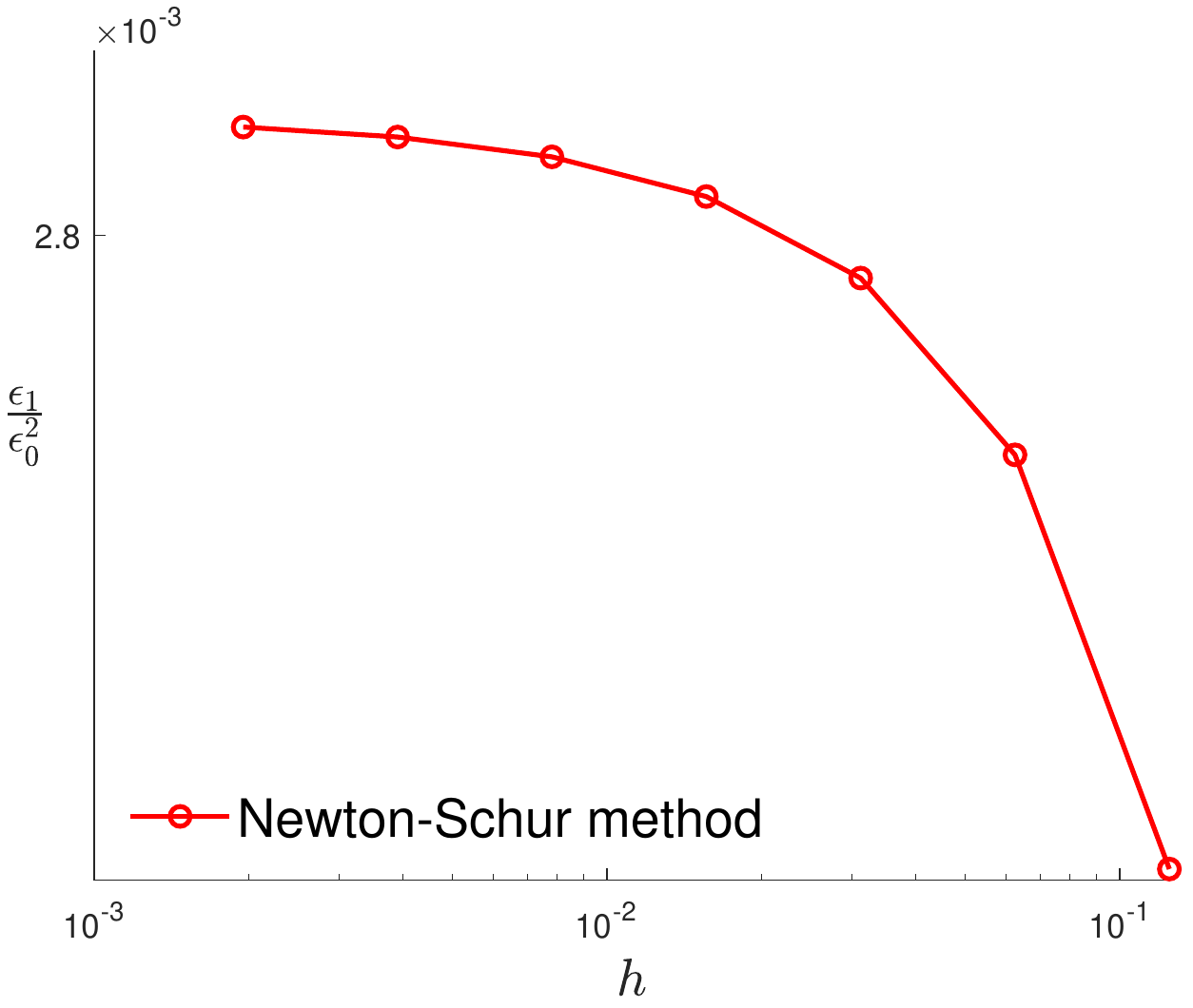}
        }
        \subfloat[Various coarse mesh sizes]{
            \includegraphics[width=\figsizeD]{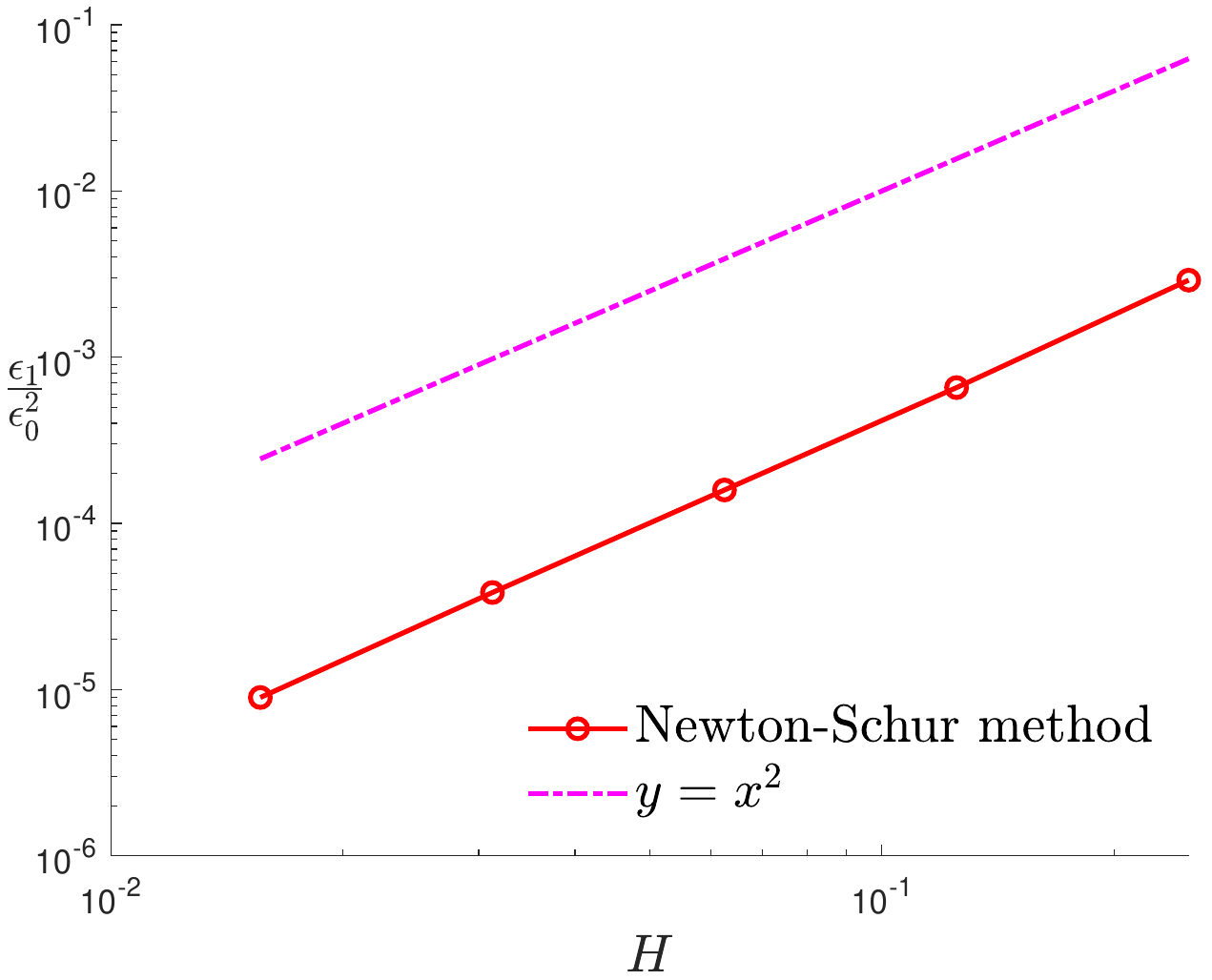}
        }
        \caption{Convergence factor for the 2D Laplacian eigenvalue problem on an L-shaped domain.}
        \label{figfac2D}
    \end{figure}

    \section{Conclusions}
    In this paper, we study the Newton-Schur method in Hilbert space and obtain some sufficient conditions for quadratic convergence. Moreover, we analyze the Newton-Schur method for symmetric elliptic eigenvalue problems discretized by the standard finite element method and non-overlapping domain decomposition method. Theoretical analysis shows that the rate of convergence is $\epsilon_{N}\leq CH^{2}(1+\ln(H/h))^{2}\epsilon^{2}$, which is supported by our numerical results.
    
    \section*{Acknowledgments}
    The authors acknowledge helpful conversations with  Zhaojun Bai in UC Davis and  Xuejun Xu in Tongji University. Chen thanks the Super-computing Center of USTC for providing computing resource and Jinwu Zhuo in Matlab. They also thank the anonymous referees for their valuable suggestions, which significantly improved the presentation of the paper.

    \bibliographystyle{siamplain}
    \bibliography{ref}

    \end{document}